\documentclass[11pt,feqno]{article}
\usepackage{amssymb,bbm,amsmath,amsthm,amscd}
\usepackage{amsmath,amsthm,amsfonts,amssymb}
\usepackage{amssymb, epsfig,amssymb, latexsym,xcolor,cite}
\usepackage{amsfonts,psfrag,color,xcolor,url,amsthm}
\usepackage{vmargin}
\usepackage{graphicx}
\usepackage{subfigure}
\usepackage{algorithm}
\usepackage[font=small,labelfont=bf,tableposition=top]{caption}
\usepackage{lineno,hyperref}
\usepackage{pgfplots}
\usepackage{verbatim}
\usepackage{algpseudocode}
\usepackage{mathtools} 
\usepackage{mathrsfs} 
\usepackage[active]{srcltx} 
\usepackage{enumerate}
\ifpdf
\DeclareGraphicsExtensions{.eps,.pdf,.png,.jpg}
\else
\DeclareGraphicsExtensions{.eps}
\fi
\DeclareUnicodeCharacter{0301}{\'{e}}

\newtheorem{theorem}{Theorem}[section]
\newtheorem{lemma}[theorem]{Lemma}

\newtheorem{definition}[theorem]{Definition}

\newtheorem{remark}[theorem]{Remark}

\usepackage[affil-it]{authblk}
\usepackage[english]{babel}
\usepackage{lineno}

\author[1]{Birgul Koc \thanks{Corresponding author}}
\author[2]{Tomás Chacón Rebollo}
\author[3]{Samuele Rubino}

\affil[1]{IMUS, Universidad de Sevilla, Spain, birkoc@alum.us.es}
\affil[2]{Departamento EDAN $\&$ IMUS, Universidad de Sevilla, Spain, chacon@us.es}
\affil[3]{Departamento EDAN $\&$ IMUS, Universidad de Sevilla, Spain, samuele@us.es}

\begin{document}
\title{Uniform Bounds with Difference Quotients for Proper Orthogonal Decomposition Reduced Order Models of the Burgers Equation}
\date{}
\maketitle

\begin{abstract}
In this paper, we prove
uniform error bounds for 
proper orthogonal decomposition (POD) reduced order modeling (ROM)
of Burgers equation, considering difference quotients (DQs), introduced in \cite{KV01}.
In particular, we study the behavior of the DQ ROM error bounds by considering $L^2(\Omega)$ and $H^1_0(\Omega)$ POD spaces and $l^{\infty}(L^2)$ and natural-norm errors. We present some meaningful numerical tests checking the behavior of error bounds.
Based on our numerical results, DQ ROM errors are several orders of magnitude smaller than noDQ ones (in which the POD is constructed in a standard way, i.e., without the DQ approach) in terms of the energy kept by the ROM basis. Further, noDQ ROM errors have an optimal behavior, while DQ ROM errors, where the DQ is added to the POD process, demonstrate an optimality/super-optimality behavior. It is conjectured that this possibly occurs because the DQ inner products allow the time dependency in the ROM spaces to make an impact.
\end{abstract}

{\bf{Keywords:}} Difference Quotients, Proper Orthogonal Decomposition, Reduced Order Models, Error Analysis, Optimality.

\section{Introduction} \label{sec:intro}
Reduced order models (ROMs) are one of the most popular low-dimensional surrogate models to obtain the numerical simulation of linear and nonlinear systems \cite{crommelin2004strategies,gunzburger2017ensemble,hesthaven2015certified,HLB96,noack2011reduced,perotto2017higamod,quarteroni2015reduced,sapsis2009dynamically,stefanescu2015pod,taira2019modal,fernandez2018computational,holmes2012turbulence,koc2019commutation,mou2021data}. To build a low dimensional ROM, one can use the following most popular frameworks such as proper orthogonal decomposition (POD)~\cite{HLB96,noack2011reduced,fareed2018note,volkwein2013proper,fukunaga1990introduction,banks1998reduced,afanasiev2001adaptive,kunisch1999control,ly2001modeling,banks2000nondestructive,ballarin2015supremizer,azaiez2021cure,iollo2000stability,bergmann2009enablers,weller2010numerical}, reduced basis methods (RBM) \cite{rozza2007stability,quarteroni2014reduced,chacon2016reduced}, empirical interpolation method (EIM), and discrete empirical interpolation method (DEIM) \cite{rowley2005model,antil2014application,drmac2016new,rebollo2017certified}. In this work, we specifically use the POD framework to build the ROM.

The ROM error analysis for the parabolic problems has been worked in \cite{iliescu2014are,KV01,kunisch2002galerkin,rubino2018streamline,singler2014new,locke2020new}. \textit{Difference quotients (DQs)} (i.e., scaled snapshots of the form $(u^{n} - u^{n-1}) / \Delta t, \, n = 1, \ldots, N$) were proposed by Kunisch and Volkwein in Remark 1 of~\cite{KV01} as a means to achieve time discretization optimality. The effect of the DQs in linear applications and the ROM error analysis for the parabolic problems are investigated in \cite{chen2016multilevel,li2022efficient} and \cite{iliescu2014are,KV01,locke2020new,kean2020error}, respectively. The authors in \cite{iliescu2014are} provide the DQ ROM numerical results
for the Burgers equation without providing any numerical analysis for the results. Thus, we emphasize that, to our knowledge, the ROM error analysis considering the DQs for nonlinear problems has never been proven. 

The main aim of the paper is to provide an error analysis for the ROM approximation of Burgers equations~\eqref{eqn:burgers} considering DQ, $L^2(\Omega)$ and $H_0^1(\Omega)$ POD space frameworks and the $l^{\infty}(L^2(\Omega))$ and natural-norms. We provide uniform ROM error bounds and present numerical tests in which we observe a much smaller error of the DQ ROM errors than the noDQ ones, with respect to energy kept in the ROM basis. This indicates that the DQ ROM keeps better-selected information for the same amount of energy due to the difference quotients.

The rest of the paper is organized as follows: In Section~\ref{sec:POD}, we briefly describe the noDQ/DQ POD
methodology and provide error bounds for the POD projection error. DQ Crank-Nicolson (CN) ROM for the Burgers equation~\eqref{eqn:burgers} is defined in Section~\ref{sec:DQ_err_estimate}. The error analysis of the DQ CN ROM is discussed in Section \ref{sec:dq_error_analysis}. 
Specifically, in Section \ref{sec:behavior_error_bounds}, we classify the optimality type for the different ROM discretization errors.
Furthermore, in Sections~\ref{sec:uniform_max_l2_error} and \ref{sec:uniform_natural_norm_error}, we provide $l^{\infty}(L^2)$ and natural-norm, i.e., $l^{\infty}(L^2) \cap l^{2}(H_0^1) $ DQ ROM error bounds considering $L^2(\Omega)$ and $H^1_0(\Omega)$
POD bases and their optimality behavior, respectively. As a mathematical model, we use viscous Burgers equation:
\begin{eqnarray} \label{eqn:burgers}
\begin{cases}
\quad	u_t -\nu u_{xx} + u u_x =  f \, , \quad x \in \Omega, \, t \in (0,T), \\
\quad	u(0,t) = u(1,t) = 0 \, , \quad t \in (0,T), \\
\quad u(x,0) = u_0(x)\, , \quad x \in \Omega.
\end{cases}
\end{eqnarray}
In Section~\ref{sec:numerical_results}, we provide ROM error bounds considering noDQ, DQ framework, $L^2(\Omega)$, $H_0^1(\Omega)$ POD spaces, and $l^{\infty}(L^2)$, natural-norm errors. Specifically, in Sections~\ref{sec:nodq_rom_numerics} and \ref{sec:dq_rom_numerics}, we numerically discuss the optimality behavior of the noDQ and DQ ROM errors, respectively, considering $L^2(\Omega)$, $H_0^1(\Omega)$ POD spaces, and $l^{\infty}(L^2)$, natural-norm errors. In Section~\ref{sec:nodq_dq_cn_pod_rom_comparison_numerics}, we numerically compare and discuss the noDQ and DQ ROM errors considering both POD spaces and norm errors. Finally, Section~\ref{sec:conclusions} presents the conclusions and future research directions.

\section{Proper Orthogonal Decomposition (POD)} \label{sec:POD}
This section builds a general POD framework with/without DQs. The construction of the noDQ/DQ POD basis is straightforward and can be summarized in the following steps: (i) We collect a snapshot data set S:=$\{u^0, u^1,..., u^N\}$ that is contained in a real Hilbert space $\mathcal{H}$ and by solving \eqref{eqn:burgers} for equispaced
parameter $t_{j}=j\Delta t, \forall j=0,..,N$ where $\Delta t =T/N$. (ii) We obtain noDQ/DQ orthonormal POD basis functions (usually called POD modes), i.e., $\{\varphi_1,...,\varphi_r\} \subset S$ 
with fixed $r>0$ value by solving the following generalized minimization problem:
\begin{align}\label{eqn:min_problem_general}
\min_{(\varphi_i, \varphi_j)_{\mathcal{H}}= \delta_{ij}} \|u-P_ru\|^2_{*}
\end{align}
where the $*$ norm is defined as
\begin{align} \label{eqn:pod_defn}
\|v\|_{*}:=  \frac{1}{\sqrt{M_{\xi}}} \Bigg( \sum_{i=0}^{N} \|v^i\|^2_{\mathcal{H}} + \xi \sum_{i=1}^{N} \|\partial v^i\|^2_{\mathcal{H}} \Bigg)^{1/2},
\end{align}
with the weight $M_{\xi}$ and tuning parameter $\xi$, and the DQs, i.e., $\partial v^k$, defined by 
\begin{align}\label{eqn:DQ_defn}
\partial v^k := \frac{ v^{k} - v^{k-1} }{ \Delta t },
\end{align}
and $P_r : \mathcal{H} \to \mathcal{H}$ is the orthogonal projection onto $ X^r := \mathrm{span}\{ \varphi_i \}_{i=1}^r $ given by
\begin{align}\label{eqn:POD_projection_H}
P_r u = \sum_{i=1}^r ( u, \varphi_i )_{\mathcal{H}} \, \varphi_i \, , \quad u \in \mathcal{H}.
\end{align}

Depending on the choice of the tuning parameter $\xi$ (we let $\xi$ to be 0 or 1) in \eqref{eqn:pod_defn}, \eqref{eqn:pod_basis} yields the noDQ/DQ POD. To be more specific, the choice of $\xi=0$ discards the DQ summation in \eqref{eqn:pod_defn}; thus, the minimization problem \eqref{eqn:min_problem_general} is solved for just the given snapshot data set $S$ with the weight $W_{\xi}=N+1$ to obtain the standard POD (noDQ POD) basis. On the other hand, choosing $\xi=1$ (taking into account the DQ summation) leads to \eqref{eqn:min_problem_general} to be solved for the different snapshot data set $S=\{u^0,..,u^N,\partial u^1,...,\partial u^N\}$ with $M_{\xi}=2N+1$, which results in the DQ POD basis.

In order to solve \eqref{eqn:min_problem_general}, one considers the eigenvalue problem
\begin{align} \label{eqn:eig_problem}
K v = \lambda v,
\end{align}
where $K_{ji}:=(u^j,u^i)_{\mathcal{H}}, i,j=1,..,M_{\xi}$ is the snapshot correlation matrix,
$\lambda_1 \geq \lambda_2 \geq,..., \geq \lambda_d > 0$ are the positive eigenvalues, and $v_k, k=1,...,d$, are the
associated eigenvectors. Then, the solution of \eqref{eqn:min_problem_general} is given by
\begin{align} \label{eqn:pod_basis}
\varphi_i=\frac{1}{\sqrt{\lambda_i}} ~\sum_{j=1}^{M_{\xi}} (v_i)_j ~u^j.
\end{align}

To obtain the $L^2$ or $H_0^1$ POD space framework, one needs to choose the Hilbert space $\mathcal{H}$ in \eqref{eqn:pod_defn} and \eqref{eqn:eig_problem} as either $L^2$ or $H_0^1$, respectively. For example, if one needs to create the DQ-H01 POD, one should choose $\mathcal{H}=H_0^1$ in \eqref{eqn:pod_defn} and \eqref{eqn:eig_problem} and $\xi=1$ in \eqref{eqn:pod_defn} to solve the minimization problem \eqref{eqn:min_problem_general} and eigenvalue problem \eqref{eqn:eig_problem}.

Now, we provide DQ POD approximation errors in Lemma~\ref{lemma:DQ_POD_approx_errors_Wnorm} proven in \cite{iliescu2014are} by considering different norms and projections onto $ X^r = \mathrm{span}\{ \varphi_i \}_{i=1}^r \subset \mathcal{H} $. Furthermore, we present the uniform DQ POD projection error bounds in Theorem~\ref{thm:DQ_uniform_estimates} proven in \cite{koc2021optimal}, Theorem 3.7. These results are necessary to prove DQ ROM error bounds and show their optimality behavior in Section~\ref{sec:DQ_err_estimate}. 

\begin{lemma} \label{lemma:DQ_POD_approx_errors_Wnorm}
Let $ X^r = \mathrm{span}\{ \varphi_i \}_{i=1}^r \subset \mathcal{H} $, let $ P_r : \mathcal{H} \to \mathcal{H} $ be the orthogonal projection onto $ X^r $ as defined in \eqref{eqn:POD_projection_H}, and let $d$ be the number of positive POD eigenvalues, where $\lambda_i^\mathrm{DQ}$ represents the DQ POD eigenvalues for the collection $ U^\mathrm{DQ} = \{ u^n \}_{n=0}^{N} \cup \{ \partial u^n \}_{n=1}^{N}  $ described above.  If $ W $ is a real Hilbert space with $ U^\mathrm{DQ} \subset W $ and $ R_r : W \to W $ is a bounded linear projection onto $ X^r $, then
\begin{subequations}
\begin{align}
\frac{1}{2N+1} \Bigg( \sum_{n=0}^N \left\| u^n - P_r u^n \right\|_W^2 +  \sum_{n=1}^{N} \left\| \partial u^n - P_r \partial u^n \right\|_W^2 \Bigg) &= \sum_{i=r+1}^{d} \lambda_i^\mathrm{DQ} \| \varphi_i \|_W^2,\label{eqn:DQ_POD_proj_error_W_norm1}\\
\frac{1}{2N+1} \Bigg( \sum_{n=0}^N \left\| u^n - R_r u^n \right\|_W^2 +  \sum_{n=1}^{N} \left\| \partial u^n - R_r \partial u^n \right\|_W^2 \Bigg) &= \sum_{i=r+1}^{d} \lambda_i^\mathrm{DQ} \| \varphi_i - R_r \varphi_i \|_W^2.    \label{eqn:DQ_POD_proj_error_W_norm2}
\end{align}
\end{subequations}
\end{lemma}x

\begin{theorem} \label{thm:DQ_uniform_estimates}
Let $ X^r = \mathrm{span}\{ \varphi_i \}_{i=1}^r \subset \mathcal{H} $, let $ P_r : \mathcal{H} \to \mathcal{H} $ be the orthogonal projection onto $ X^r $ as defined in \eqref{eqn:POD_projection_H}, let $d$ be the number of positive POD eigenvalues, where $\lambda_i^\mathrm{DQ}$ represents the DQ POD eigenvalues for $ U^\mathrm{DQ}$, and let $t_j:=j \Delta t, \forall j=0,...,N$ where $\Delta t = T/N$.  If $ W $ is a real Hilbert space with $ U^\mathrm{DQ} \subset W $ and $ R_r : W \to W $ is a bounded linear projection onto $ X^r $, then
\begin{subequations} 
\begin{align}
\max_{0 \leq k \leq N} \left\| u^k - P_r u^k \right\|_{\mathcal{H}}^2  &\leq  \mathfrak{C} \sum_{i = r+1}^{d} \lambda_i^\mathrm{DQ},\label{eqn:DQ_POD_uniform_bound1}\\
\max_{0 \leq k \leq N} \left\| u^k - P_r u^k \right\|_W^2  &\leq \mathfrak{C} \sum_{i = r+1}^{d} \lambda_i^\mathrm{DQ} \| \varphi_i \|_W^2,\label{eqn:POD_DQs_uniform_bound2}\\
\max_{0 \leq k \leq N} \left\| u^k - R_ru^k \right\|_W^2  &\leq  \mathfrak{C} \sum_{i = r+1}^{d} \lambda_i^\mathrm{DQ} \| \varphi_i - R_r \varphi_i \|_W^2,\label{eqn:DQ_POD_uniform_bound3}
\end{align}
\end{subequations}
where $\mathfrak{C} = 6 \max\{1,T^2 \}$.
\end{theorem}

\begin{remark}
The bounds in Lemma~\ref{lemma:DQ_POD_approx_errors_Wnorm} and Theorem~\ref{thm:DQ_uniform_estimates} are still valid if one replaces the snapshots data $\{u^0, u^1,..., u^N\}$, which are the continuous solution data in this paper as in \cite{koc2021optimal}, with finite element (FE) solutions $\{u_h^0, u_h^1,..., u_h^N\}$, where $h$ is the spatial discretization parameter. Furthermore, the data set used to generate the POD basis in \eqref{eqn:min_problem_general} should be the same as the data set used in the POD projection error bounds in Lemma~\ref{lemma:DQ_POD_approx_errors_Wnorm} and Theorem~\ref{thm:DQ_uniform_estimates}.
\end{remark}

\begin{remark} \label{remark:model_names}
The construction of all ROMs, which are used in the following sections, are obtained by using Crank-Nicolson and Galerkin time and space discretizations, respectively. However, they differ from each other based on two main criteria: noDQ/DQ and $L^2$/$H_0^1$ POD frameworks. Thus, when we label the name of the models for brevity, we drop CN, POD, and ROM acronyms, and, for clarity, we consider noDQ/DQ and $L^2$/$H_0^1$ acronyms.
\end{remark}

\section{Reduced Order Modeling (ROM)} \label{sec:DQ_err_estimate}
In this section, we present a numerical method for the Burgers equation~\eqref{eqn:burgers}, which is the proper orthogonal decomposition reduced order model.

First, we define the function space $X = H_0^1(\Omega)$ endowed with the inner product $ (u,v)_{H^1_0} = ( u_x, v_x)_{L^2} $. We take $u(\cdot,t)\in X$, $t\in [0,T]$ to be the weak solution of the weak formulation of the Burgers equation with homogeneous Dirichlet boundary conditions:
\begin{align}\label{eqn:weak_form_burgers}
(\partial_{t}u,v)_{L^2} + \nu( u_x, v_x)_{L^2} + (u u_x, v)_{L^2} = (f,v)_{L^2}\quad \forall v\in X.
\end{align}

Applying Crank-Nicolson and Galerkin 
discretizations in time and space, respectively to the weak formulation of the Burgers equation~\eqref{eqn:weak_form_burgers} results in the CN ROM: $\forall v_{r}\in X^r$,
\begin{align}\label{eqn:CN_POD_G_ROM}
\begin{aligned}
\left(\partial u_r^{n+1},v_r\right)_{L^2} + \nu( (u_r)_x^{n+1/2}, (v_r)_x)_{L^2} + (u_r^{n+1/2}  (u_r)_x^{n+1/2}, v_r)_{L^2} 
= (f^{n+1/2},v_{r})_{L^{2}},
\end{aligned}
\end{align}
where $ \partial u^{n+1}_r := ( u^{n+1}_r - u^n_r )/ \Delta t $. 

\begin{remark} \label{remark:consistency}
We use the notation $ z^{n+1/2} $ for any discrete-time function $ z $ to denote the average $z^{n+1/2} := \frac{1}{2} \left( z^{n+1} + z^n \right)$. However, for a continuous time function, we use $ f^{n+1/2} $ to denote $ f(t_n+\Delta t/2) $.
\end{remark}

\section{Error Analysis} \label{sec:dq_error_analysis}
In this section, we prove uniform error bounds for the DQ ROM for the Burgers equation~\eqref{eqn:burgers}. Specifically, we provide $l^{\infty}(L^2)$ and natural-norm ($l^{\infty}(L^2) \cap l^2(H_0^1)$) error bounds considering $L^2(\Omega)$ and $H_0^1(\Omega)$ POD bases and their optimality behavior in Section~\ref{sec:uniform_max_l2_error} and Section~\ref{sec:uniform_natural_norm_error}, respectively.

We start the analysis by applying the CN time discretization to (8), which yields the following: $\forall v \in X$,
\begin{align} \label{eqn:EE1_CN}
\left(\partial u^{n+1},v\right)_{L^2} + \nu( u_x^{n+1/2}, v_x)_{L^2} + (u^{n+1/2}  u_x^{n+1/2}, v)_{L^2} 
= (f^{n+1/2},v)_{L^{2}} + \tau_n(v),
\end{align}
with the corresponding consistency error
\begin{align}
\begin{aligned} \label{eqn:EE2_CN}
    \tau_n(v) := \left( \partial u^{n+1} - \partial_t u(t_{n}+\Delta t /2), v \right)_{L^2} + \nu \left(  u_{xx} (t_{n} + \Delta t/2) - u_{xx}^{n+1/2} ,  v \right)_{L^2}\\
     + \left(  u^{n+1/2}  u_x^{n+1/2} - u(t_{n}+ \Delta t/2)  u_x(t_{n} + \Delta t/2) , v \right)_{L^2}.
\end{aligned}
\end{align}

The consistency error~\eqref{eqn:EE2_CN} does not depend on the $f$ term because of Remark~\ref{remark:consistency}. Furthermore, we assume the following regularity conditions on the continuous
solution $u$ and the terms in \eqref{eqn:EE2_CN}:
\begin{subequations}
\begin{align}
u & \in L^\infty(H_0^1(\Omega)),
\label{eqn:uniform_bound_cont_soln} \\
u_{ttt}, \:  (u_{tt})_{xx}, \: (u  (u_{tt})_x & + u_{tt}  u_x) \in L^{2}(0,T;L^{2}(\Omega)). \label{eqn:EE3_CN}
\end{align}
\end{subequations}

Now, we define the regularity constants, which are the bounds for the terms in \eqref{eqn:EE3_CN} as
\begin{align}\label{eqn:EE4_CN}
\begin{aligned}
I_{n,1}(u) &:= \| u_{ttt} \|_{L^2(t_n,t_{n+1};L^2)} + \|  (u_{tt})_{xx} \|_{L^2(t_n,t_{n+1};L^2)}  \\
& + \|u  (u_{tt})_x + u_{tt}  u_x \|_{L^2(t_n,t_{n+1};L^2)},\\
I_n(u) &:= \| u_{ttt} \|^2_{L^2(t_n,t_{n+1};L^2)} + \|  (u_{tt})_{xx} \|^2_{L^2(t_n,t_{n+1};L^2)} \\
& +  \|u  (u_{tt})_x +u_{tt}  u_x \|^2_{L^2(t_n,t_{n+1};L^2)},\\
I(u) &:= \| u_{ttt} \|^2_{L^2(0,T;L^2)} + \| (u_{tt})_{xx} \|^2_{L^2(0,T;L^2)} \\
&+ \|u  (u_{tt})_x + u_{tt}  u_x  \|^2_{L^2(0,T;L^2)}.
\end{aligned}
\end{align}

Now, we subtract \eqref{eqn:CN_POD_G_ROM} from \eqref{eqn:EE1_CN} by choosing $v=v_r$ in \eqref{eqn:EE1_CN} (since $X^r \subset X$), and label the discretized error $e^{n+1}:= u^{n+1} -u_r^{n+1}$. Then, one gets the following error equation: $\forall v_{r}\in X^r$,
\begin{align} \label{eqn:EE5_CN}
\begin{aligned}
\left(\partial e^{n+1}, v_r\right)_{L^2} + \nu (  e_x^{n+1/2}, (v_r)_x)_{L^2} \, + \, & (u^{n+1/2}  u_x^{n+1/2},v_r)_{L^2} \\
- \, & \, (u_r^{n+1/2} (u_r)_x^{n+1/2}, v_r)_{L^2}
= \tau_n(v_{r}) .
\end{aligned}
\end{align}

Then, we split the discretized error $e^{n+1}$ into two parts as
\begin{align}\label{eqn:EE6_CN}
\begin{aligned}
e^{n+1} = u^{n+1}-u^{n+1}_r &= (u^{n+1}-w_r^{n+1})-(u^{n+1}_r-w_r^{n+1}) \\
& =\eta^{n+1}-\phi^{n+1}_r,
\end{aligned}
\end{align}
where $w_r^{n+1}:=R_r u^{n+1}$ is chosen as the Ritz projection of $u^{n+1}$ on $X^r$ (for different analyses, $w_r$ could be chosen differently), which is defined as 
\begin{align} \label{eqn:ritz_proj}
\big((u-R_r u)_x, (v_r)_x\big)_{L^2}=0\quad \forall v_{r}\in X^r.
\end{align}
The POD projection error ($\eta$) and the discretization error ($\phi$) in \eqref{eqn:EE6_CN} are defined as
\begin{subequations}
\begin{align}
\eta^{n+1}:=u^{n+1}-w_r^{n+1}, \label{eqn:pod_proj_err} \\
\phi^{n+1}_r:=u^{n+1}_r-w_r^{n+1}. \label{eqn:disc_err}
\end{align}
\end{subequations}

During the analysis, we need a standard stability estimate of $u_r^n$ for the CN scheme, which is
\begin{align}  \label{eqn:stability_est_rom_soln}
\max_{ 0 \leq n \leq N  }\|u_r^{n}\|_{L^2}  \leq \mathbbm{C}.
\end{align}

Now, by using the error splitting~\eqref{eqn:EE6_CN}, one can rewrite the error equation~\eqref{eqn:EE5_CN} as 
\begin{align} \label{eqn:EE7_CN}
\begin{aligned}
(\partial \phi^{n+1}_r ,v_r)_{L^2}+\nu( (\phi_r)_x^{n+1/2}, (v_r)_x )_{L^2} 
= (\partial \eta^{n+1},v_r)_{L^2}+\nu( \eta_x^{n+1/2}, (v_r)_x )_{L^2} \\
+ (u^{n+1/2} u_x^{n+1/2}, v_r)_{L^2} -  (u_r^{n+1/2} (u_r)_x^{n+1/2}, v_r)_{L^2}  - \tau_n\left( v_{r} \right).
\end{aligned}
\end{align}

Then, \eqref{eqn:ritz_proj} leads to $(\eta_x, (v_r)_x)_{L^2}=0, \, \forall v_r \in X^r$; thus, the second term $(\eta_x^{n+1/2}, (v_r)_x)_{L^2}$ on the right-hand side of \eqref{eqn:EE7_CN} vanishes. We continue the error analysis by choosing $v_r:=\phi_r^{n+1/2}$, then \eqref{eqn:EE7_CN} is rewritten as
\begin{align} \label{eqn:EE10_CN}
\begin{aligned}
\frac{1}{2\Delta t} \Big(\| \phi_r^{n+1}\|^2_{L^2}-\| \phi_r^{n}\|^2_{L^2} \Big) + \nu \| (\phi_r)_x^{n+1/2}\|^2_{L^2} = (\partial \eta^{n+1}, \phi_r^{n+1/2})_{L^2} \\
+ (u^{n+1/2} u_x^{n+1/2}, \phi_r^{n+1/2} )_{L^2} -  (u_r^{n+1/2} (u_r)_x^{n+1/2}, \phi_r^{n+1/2})_{L^2}  - \tau_n( \phi_r^{n+1/2} ).
\end{aligned}
\end{align}

Now, we individually bound the terms in \eqref{eqn:EE10_CN}. During the analysis, $C$ is a generic constant that only depends on the data. By using the Cauchy-Schwarz inequality and the Young's inequality, the first term on the right-hand side of \eqref{eqn:EE10_CN}, can be bounded as
\begin{align} \label{eqn:EE11_CN}
\begin{aligned}
(\partial \eta^{n+1}, \phi_r^{n+1/2})_{L^2} &\leq \|\partial \eta^{n+1} \|_{L^2} \|\phi_r^{n+1/2}\|_{L^2} \\
& \leq \frac{1}{4} \, \|\partial \eta^{n+1} \|^2_{L^2} + \| \phi_r^{n+1/2}\|^2_{L^2}.
\end{aligned}
\end{align}

Next, we arrange the nonlinear terms in \eqref{eqn:EE10_CN} by adding and subtracting the term $(u_r^{n+1/2} u_x^{n+1/2}, \phi_r^{n+1/2})_{L^2}$. Then, we rewrite the nonlinear terms as
\begin{align} \label{eqn:EE12_CN}
    \begin{aligned}
&( u^{n+1/2} u_x^{n+1/2}, \phi_r^{n+1/2} )_{L^2} -  (u_r^{n+1/2} (u_r)_x^{n+1/2}, \phi_r^{n+1/2} )_{L^2} \\
& = ( (u-u_r)^{n+1/2} u_x^{n+1/2}, \phi_r^{n+1/2} )_{L^2} +  (u_r^{n+1/2} (u-u_r)_x^{n+1/2}, \phi_r^{n+1/2})_{L^2} \\
& = (\eta^{n+1/2} u_x^{n+1/2}, \phi_r^{n+1/2} )_{L^2} - (\phi_r^{n+1/2} u_x^{n+1/2}, \phi_r^{n+1/2} )_{L^2} \\
& +  (u_r^{n+1/2} \eta_x^{n+1/2}, \phi_r^{n+1/2})_{L^2} -(u_r^{n+1/2} (\phi_r)_x^{n+1/2}, \phi_r^{n+1/2})_{L^2}.
\end{aligned}
\end{align}

Now, we individually bound nonlinear terms in \eqref{eqn:EE12_CN}.
By using the Hölder's inequality, the regularity condition of the continuous solution in \eqref{eqn:uniform_bound_cont_soln}, and Young's inequality, we can bound the first term in \eqref{eqn:EE12_CN} as
\begin{align} \label{eqn:EE13_CN}
\begin{aligned}
(\eta^{n+1/2} u_x^{n+1/2}, \phi_r^{n+1/2} )_{L^2} & \leq \|\eta^{n+1/2}\|_{L^2} \|u_x^{n+1/2} \|_{L^\infty} \|\phi_r^{n+1/2}\|_{L^2} \\
& \leq C \|\eta^{n+1/2}\|_{L^2}  \|\phi_r^{n+1/2}\|_{L^2} \\    
& \leq C \, \|\eta^{n+1/2}\|^2_{L^2} + \|\phi_r^{n+1/2}\|^2_{L^2}.
\end{aligned}
\end{align}

For the second term in \eqref{eqn:EE12_CN}, we use the Hölder's inequality and the regularity condition of the continuous solution in \eqref{eqn:uniform_bound_cont_soln} as
\begin{align} \label{eqn:EE14_CN}
\begin{aligned}
(\phi_r^{n+1/2} u_x^{n+1/2}, \phi_r^{n+1/2} )_{L^2} 
& \leq \|\phi_r^{n+1/2}\|_{L^2} \| u_x^{n+1/2} \|_{L^\infty } \|\phi_r^{n+1/2}\|_{L^2} \\
& \leq C \|\phi_r^{n+1/2}\|^2_{L^2} .
\end{aligned}
\end{align}

For the third term in \eqref{eqn:EE12_CN}, we use the Hölder's inequality, the standard stability estimate of $u_r^n$ in $l^{\infty}(L^2)$ for CN scheme in \eqref{eqn:stability_est_rom_soln}, the Sobolev embedding, and the Young's inequality as
\begin{align} \label{eqn:EE15_CN}
\begin{aligned}
(u_r^{n+1/2} \eta_x^{n+1/2}, \phi_r^{n+1/2})_{L^2} & \leq \|u_r^{n+1/2}\|_{L^2} \| \eta_x^{n+1/2}\|_{L^2} \| \phi_r^{n+1/2}\|_{L^\infty} \\
& \leq C  \| \eta_x^{n+1/2}\|^2_{L^2} + C_1 \|(\phi_r)_x^{n+1/2}\|^2_{L^2}.
\end{aligned}
\end{align}

Finally, for the last nonlinear term in \eqref{eqn:EE12_CN}, we use the Hölder's inequality, the stability estimate of $u_r^n$ in \eqref{eqn:stability_est_rom_soln}, the Agmon's inequality (eq. after (45) in \cite{KV01}):
\begin{align*}
\|\varphi \|_{L^\infty} \leq C \| \varphi \|^{1/2}_{L^2} \, \| \varphi_x \|^{1/2}_{L^2}, \quad \forall \varphi \in H_0^1,
\end{align*}
and the Young's inequality ($p=4, q =4/3$), then we get
\begin{align} \label{eqn:EE16_CN}
\begin{aligned}
(u_r^{n+1/2} (\phi_r)_x^{n+1/2}, \phi_r^{n+1/2})_{L^2} & \leq \|u_r^{n+1/2}\|_{L^2} \| (\phi_r)_x^{n+1/2}\|_{L^2} \| \phi_r^{n+1/2}\|_{L^{\infty}} \\
& \leq C \| \phi_r^{n+1/2}\|^{1/2}_{L^2}  \| (\phi_r)_x^{n+1/2}\|^{3/2}_{L^2} \\
& \leq C  \| \phi_r^{n+1/2}\|^2_{L^2} + C_2 \| (\phi_r)_x^{n+1/2}\|^2_{L^2}.
\end{aligned}
\end{align}

To bound the consistency error \eqref{eqn:EE2_CN}, we use the Taylor's theorem, the Young's inequality, and the property $(a+b+c)^2 \leq 3 (a^2+b^2+c^2)$ for $I_{n,1}(u)$ in \eqref{eqn:EE4_CN}, then we have
\begin{align} \label{eqn:EE17_CN}
\begin{aligned}
\tau_n(\phi_r^{n+1/2}) & \leq  \Delta t^{3/2} I_{n,1}(u)  \, \|\phi_r^{n+1/2}\|_{L^2} \\
& \leq \frac{3}{4} \, \Delta t^3 I_n(u) + \| \phi_r^{n+1/2}\|^2_{L^2}.
\end{aligned}
\end{align}

Now, we choose coefficient $C_1 = C_2 = \nu/4$, and insert all bounds \eqref{eqn:EE11_CN}-\eqref{eqn:EE17_CN} into \eqref{eqn:EE10_CN}, then
\begin{align} \label{eqn:EE18_CN}
\begin{aligned}
\frac{1}{2\Delta t}\Big(  \| \phi_r^{n+1}\|^2_{L^2} - \| \phi_r^{n}\|^2_{L^2} \Big) + \frac{\nu}{2} \| (\phi_r)_x^{n+1/2}\|^2_{L^2} \leq C \Big[\|\phi_r^{n+1/2}\|^2_{L^2} + \|\partial \eta^{n+1}\|^2_{L^2} \\  + \|\eta^{n+1/2}\|^2_{L^2} +  \|\eta_x^{n+1/2}\|^2_{L^2} + \Delta t^3 I_n(u) \Big].
\end{aligned}
\end{align}

To derive a valid error bound from \eqref{eqn:EE18_CN}, there should be a relation between the time step $\Delta t$ and the viscosity coefficient $\nu$ that is explained in the following lemma.
\begin{lemma} \label{lemma:uniform_bound}
Let $\Delta t < \frac{4 \, \mathcal{C} \, \nu^3}{27}$, then it holds
\begin{align} \label{eqn:EE22_CN}
\begin{aligned}
(1-C\Delta t) \| \phi_r^{n+1}\|^2_{L^2} + \nu \Delta t \|(\phi_r)_x^{n+1/2}\|^2_{L^2} \leq  \big(1+ C \Delta t \big) \|\phi_r^n\|^2_{L^2} +  C \Delta t \Big[  \|\partial \eta^{n+1}\|^2_{L^2} \\
+  \|\eta^{n+1/2}\|^2_{L^2} + \| \eta_x^{n+1/2}\|^2_{L^2}  + \Delta t^3 I_n(u) \Big],
\end{aligned}
\end{align}
where the constant $\mathcal{C}$ above is independent of all discretization parameters but depends on the problem data.
\end{lemma}

\begin{proof}
First, bounding the term $\|\phi_r^{n+1/2}\|^2_{L^2}$ yields $\frac{1}{2} \Big(\|\phi_r^{n+1}\|^2_{L^2} + \|\phi_r^n\|^2_{L^2} \Big)$ in \eqref{eqn:EE18_CN}. Then, if one computes the constant coefficients in \eqref{eqn:EE11_CN}-\eqref{eqn:EE17_CN}, the generic constant $C$ in \eqref{eqn:EE18_CN} will be
\begin{align} \label{eqn:generic_const}
C:= \max \{ \frac{1}{4}, \frac{\|u^{n+1/2}_x \|^2_{L^\infty}}{4}, \frac{\|u^{n+1/2}_x \|_{L^\infty}}{2}, \frac{\mathbbm{C}^2}{\nu}, \frac{27 \, \mathbbm{C}^4}{8\nu^3}, \frac{3}{4}, \frac{3}{2} \}.
\end{align}
Then the constant $C$ does not depend on the ROM solution. Without loss of generality, in this paper, we assume that $\frac{27 \,  \mathbbm{C}^4}{8 \, \nu^3}$ dominates the terms in \eqref{eqn:generic_const}.
Finally, multiplying the resulting equation with $2\Delta t$, one obtains \eqref{eqn:EE22_CN} with the constraint $\Delta t < \frac{4 \, \mathcal{C} \, \nu^3}{27}$, where $\mathcal{C} = \mathbbm{C}^{-4}$. 
\end{proof}

\begin{remark}
One can bound the term $\| (\phi_r)^{n+1/2}_x \|^2_{L^2}$ in \eqref{eqn:EE15_CN}-\eqref{eqn:EE16_CN} by using FE inverse
estimates, as in Theorem 9.2 in \cite{rebollo2014mathematical}.

\begin{remark}
For any $\phi^{n+1}_r \in X^h$, which is the FE space that contains $X^r$, then the following FE inverse estimate holds:
\begin{align} \label{eqn:fe_inverse_estimate}
\|(\phi^{n+1}_r)_x\|_{L^2} \le C\, h^{-1} \|\phi^{n+1}_r\|_{L^2}.
\end{align}
\end{remark}
However, by using the FE inverse estimate to bound the term $\| (\phi_r)^{n+1/2}_x \|^2_{L^2}$ in \eqref{eqn:EE15_CN}-\eqref{eqn:EE16_CN}, the generic constant $C$ in \eqref{eqn:EE22_CN} depends on the space discretization, and we eventually loose $h$ convergence order, assuming uniformly regular grid.

Furthermore, the bound in \eqref{eqn:EE22_CN} obtained by using the Galerkin method leads to a restriction on the time step $\Delta t$ (see Lemma~\ref{lemma:uniform_bound}). Using stabilized methods would allow relaxing this restriction. In this paper, our concern is analyzing the error estimates optimality with respect to the different POD setting strategies, so we will consider moderate values of $\nu$.
\end{remark}

\begin{remark}
In this paper, we construct the POD basis by using the snapshots, which are the FE solutions; thus, $\eta^n$ bounds in Lemma~\ref{lemma:DQ_POD_approx_errors_Wnorm} and Theorem~\ref{thm:DQ_uniform_estimates} should be expressed in terms of the FE solution data. We next bound the projection error $\eta^n$ in terms of the finite element solution.
\end{remark}

We next bound the projection error $\eta^n$ in terms of the finite element solution.
\begin{lemma} \label{lemma:ritz_proj_u_uh}
Let $R_r u$ and $R_r u_h$ be the Ritz projection, which is defined in \eqref{eqn:ritz_proj}, of the continuous solution $u$ and FE solution $u_h$, respectively, then the following estimates hold
\begin{subequations}
\begin{align} 
\frac{1}{N+1} \sum_{n=0}^{N} \| u^n - R_r u^n \|^2_{W} & \le C \, ( h^{2l} + \Delta t^4 ) + \frac{1}{N+1} \sum_{n=0}^{N} \| u_h^n - R_r u_h^n \|^2_{W}  \label{eqn:eta_update1a} \\ 
\frac{1}{N} \sum_{n=1}^{N} \| \partial(u^n - R_r u^n) \|^2_{L^2} & \le C \, ( h^{2l+1} + \Delta t^3 ) + \frac{1}{N} \sum_{n=1}^{N} \| \partial(u_h^n - R_r u_h^n) \|^2_{L^2}, \label{eqn:eta_update1b}
\end{align}
\end{subequations}
where $h$ and $\Delta t$ are space and time discretization parameters, respectively, and $l$ is the FE interpolation order. 
\begin{proof}
We start proving with $W=H_0^1$ in \eqref{eqn:eta_update1a}. By using the definition of the Ritz projection \eqref{eqn:ritz_proj}, we have
\begin{align} \label{eqn:eta_update2}
\begin{aligned}
\| \nabla (u^n - R_r u^n) \|^2_{L^2} &= (\nabla (u^n - R_r u^n), \nabla (u^n - R_r u^n) ), \\
&= (\nabla (u^n - R_r u^n), \nabla (u^n - R_r u_h^n) ), \\
& \leq \| \nabla (u^n - R_r u^n) \|_{L^2} \, \| \nabla (u^n - R_r u_h^n) \|_{L^2}.
\end{aligned}
\end{align}
Then, by adding and subtracting the FE solution $u_h^n$ in \eqref{eqn:eta_update2} and summing the resulting inequality from $n=0$ to $N$, we get
\begin{align}
\begin{aligned}
\| \nabla (u^n - R_r u^n) \|^2_{L^2} & \leq  \| \nabla (u^n - u_h^n) \|^2_{L^2} + \| \nabla (u_h^n - R_r u_h^n) \|^2_{L^2}, \\
\frac{1}{N+1} \sum_{n=0}^{N} \| \nabla (u^n - R_r u^n) \|^2_{L^2} & \leq C \, (h^{2l} + \Delta t^4 ) + \frac{1}{N+1} \sum_{n=0}^{N} \| \nabla (u_h^n - R_r u_h^n) \|^2_{L^2}.
\end{aligned}
\end{align}
For $W=L^2$ case, we still revisit the definition of the Ritz projection \eqref{eqn:ritz_proj} and choose $v_r=R_r u$, then we have
\begin{align} 
\| \nabla R_r u \|_{L^2} \le \| \nabla u \|_{L^2} \label{eta_update3}.
\end{align}
Then, by adding and subtracting the FE solution $u_h^n$ in $\|u^n-R_ru^n\|_{L^2}$, applying the Poincaré inequality, the relation in \eqref{eta_update3}, and summing from $n=0$ to $N$ give
\begin{align} \label{eta_update3a}
\begin{aligned}
\|u^n-R_ru^n\|^2_{L^2} &\le \|u^n-u_h^n\|^2_{L^2} + \|R_r(u^n-u_h^n)\|^2_{L^2} + \|u_h^n-R_ru_h^n\|^2_{L^2},  \\
&\le \|u^n-u_h^n\|^2_{L^2} + C_p \, \| \nabla R_r(u^n-u_h^n)\|^2_{L^2} + \|u_h^n-R_ru_h^n\|^2_{L^2}, \\
\frac{1}{N+1} \sum_{n=0}^{N} \|u^n-R_ru^n\|^2_{L^2} & \le C \, (h^{2l} + \Delta t^4 ) + \frac{1}{N+1} \sum_{n=0}^{N} \|u_h^n-R_ru_h^n\|_{L^2}.
\end{aligned}
\end{align}
Finally, to bound the term $\|\partial(u^n-R_ru^n)\|$, we follow similar steps as in \eqref{eta_update3a} and get the following:
\begin{align} \label{eta_update4a}
\frac{1}{N} \sum_{n=1}^{N} \|\partial(u^n-R_ru^n)\|^2_{L^2} & \le C \, (h^{2l} + \Delta t^3 ) + \frac{1}{N}\sum_{n=1}^{N} \|\partial(u_h^n-R_ru_h^n)\|_{L^2}.
\end{align} 
\end{proof}
\end{lemma}

Now, in the following two sections, by using Lemma~\ref{lemma:uniform_bound}, we continue to derive and discuss DQ ROM errors in different norms considering $L^2(\Omega)$ and $H_0^1(\Omega)$ POD bases that will lead to different consistency error estimates.

\subsection{Behavior of Error Bounds} \label{sec:behavior_error_bounds}
In Sections \ref{sec:uniform_max_l2_error} and \ref{sec:uniform_natural_norm_error}, we discuss the behavior of the different DQ error bounds with respect to ROM discretization. Before proceeding with these results, which are derived in the next sections, we provide definitions to distinguish whether ROM discretization error is suboptimal or optimal and to classify the types of the optimality/suboptimality behavior of the ROM discretization errors.

The behavior of a pointwise error bound depends on both the space $\mathcal{H}$ for the POD basis and the space $W$ for the pointwise error norm. The expected pointwise error bounds have the structure:
\begin{align}  \label{eqn:pointwise_error_bound}
\max_{ 0 \leq k \leq N  }\|e^{k}\|^{2}_{W}  \leq  C (\Lambda_r + \Lambda^{0}_r + \zeta(\Delta t) + \xi(h)),
\end{align}
where $\Delta t$ and $h$ are the time and spatial discretization parameters, and $\Lambda_r$, $\Lambda^{0}_r$, $\zeta(\Delta t)$, and $\xi(h)$ represents the ROM discretization error, the ROM discretization error for the initial condition, time discretization error, and spatial discretization error, respectively.

Since we are interested in the behavior of a pointwise error bound only with respect to the ROM discretization, in the following sections, we provide the definition and types of optimality and the definition of suboptimality in ROM discretization error sense.

A ROM discretization error, i.e., $\Lambda_r$ is called \textit{optimal} if it is bounded by $\Lambda_r^{*}$, $\Lambda_r^{I}$ or $\Lambda_r^{II}$ in \eqref{eqn:ROMerror_truly_optimal}-\eqref{eqn:ROMerror_optimalII} in Definition \ref{defn:optimal}. Depending on how the ROM discretization error is bounded (see Definition~\ref{defn:optimal}), the type of the optimality differs such as \textit{truly optimal}, \textit{optimal-I} discussed in \cite{iliescu2014variational} or \textit{optimal-II} discussed in \cite{koc2021optimal}.

\begin{definition} \label{defn:optimal}
Let $ X^r \subset \mathcal{H} $ be the span of the first $ r $ POD modes, and assume $ X^r $ is also contained in $ W $.  Let $ P_r : \mathcal{H} \to \mathcal{H} $ be the orthogonal POD projection onto $ X^r $, and let $ \Pi^W_r : W \to W $ be the $ W $-orthogonal projection onto $X^r $.  Also, let $d$ be the number of positive POD eigenvalues. Then, the ROM discretization error, i.e., $\Lambda_r$, is
\begin{subequations}
\begin{align}
\textbf{truly optimal:} \quad \Lambda_r &\leq C \Lambda_r^{*}, \quad \Lambda_r^{*}:=\Big( \max_{1 \leq k \leq N} \| u^k - \Pi^W_r u^k \|_W^2 \Big), \label{eqn:ROMerror_truly_optimal} \\
\textbf{optimal-I:} \quad \Lambda_r &\leq C \Lambda_r^{I}, \quad \Lambda_r^{I}:=\Big( \sum_{i = r+1}^{d} \lambda_i \| \varphi_i \|_W^2 \Big), \label{eqn:ROMerror_optimalI} \\
\textbf{optimal-II:} \quad \Lambda_r &\leq C \Lambda_r^{II}, \quad \Lambda_r^{II}:=\Big( \sum_{i = r+1}^{d} \lambda_i \| \varphi_i - \Pi^W_r \varphi_i \|_W^2 \Big), \label{eqn:ROMerror_optimalII}
\end{align}
\end{subequations}
where the constant $ C $ above should be independent of all discretization parameters but may depend on the solution data and the problem data.
\end{definition}

\begin{remark} \label{remark:suboptimality}
If the given ROM discretization error does not meet any criteria in Definition~\ref{defn:optimal} or the constant $C$ depends on the ROM discretization parameter such as $r$, then it is called suboptimal.
\end{remark}

In Sections \ref{sec:uniform_max_l2_error} and \ref{sec:uniform_natural_norm_error}, we
consider four possibilities: we used $\mathcal{H}=L^2$ or $\mathcal{H}=H_0^1$ for the POD basis, and
we use $W=L^2$ or $W=H_0^1$ for the error norm and use Definition \ref{defn:optimal} and Lemma \ref{lemma:pod_inverse_estimate}, i.e., POD inverse estimates, which was proved in Lemma 2 and Remark 2 in \cite{KV01}, to determine the pointwise error bounds behavior.

To state these inverse estimates, let $M_r \in \mathbbm{R}^{r \times r}$ with $M_{ij}:(\varphi_j, \varphi_i)_{L^2}$ be the POD mass matrix and $S_r \in \mathbbm{R}^{r \times r}$ with $S_{ij}:(\nabla \varphi_j, \nabla \varphi_i)_{L^2}$ be the POD stiffness matrix. Let $\| \cdot \|_{2}$ denote the matrix 2-norm.
\begin{lemma} \label{lemma:pod_inverse_estimate}
For all $v_r \in X^r$, the following POD inverse estimates hold:
\begin{subequations}
\begin{align}
\|\nabla v_r\|_{L^2} & \leq C^{L^2}_{inv}(r) \, \|v_r\|_{L^2}, \quad \text{for the $L^2$-POD}, \label{eqn:l2_pod_inv_est} \\
\|\nabla v_r\|_{L^2} &\leq C^{H_0^1}_{inv}(r) \, \|v_r\|_{L^2}, \quad \text{for the $H_0^1$-POD}, \label{eqn:h01_pod_inv_est}
\end{align}
\end{subequations}
where $C^{L^2}_{inv}(r)=\sqrt{\|S_r\|_2}$ and $C^{H_0^1}_{inv}(r)=\sqrt{\|M_r^{-1}\|_2}$.
\end{lemma}

\subsection{The $l^\infty(L^2)$ Error Estimates} \label{sec:uniform_max_l2_error}
In this section, we provide the $l^\infty(L^2)$ error estimates for the DQ ROM~\eqref{eqn:CN_POD_G_ROM}, considering both $L^2(\Omega)$ and $H_0^1(\Omega)$ POD spaces. Furthermore, we discuss the behavior of the $l^\infty(L^2)$ DQ-L2 and DQ-H01 error in 
Theorem~\ref{theorem:linfinity_L2_dq_l2_error_bound} and Theorem~\ref{theorem:linfinity_L2_dq_h01_error_bound}, respectively.

\begin{theorem} \label{theorem:linfinity_L2_dq_l2_error_bound}
Assume that $\Delta t \le \frac{2 \, \mathcal{C} \,\nu^3}{27}$, then the $l^\infty(L^2)$ DQ-L2 error is bounded by
\begin{align}   \label{eqn:linfinity_L2_dq_l2_error_bound}
\begin{aligned}
\max_{ 0 \leq k \leq N  }\|e^{k}\|^{2}_{L^2}  \leq  &C \Big[\| \phi_r^{0}\|^2_{L^2} + \sum_{i = r+1}^{d} \lambda_i^\mathrm{DQ} \Big( \| \varphi_i - R_r \varphi_i \|^2_{L^2} + \| (\varphi_i - R_r \varphi_i)_x \|^2_{L^2} \Big) \\
&+ (h^{2l} + \Delta t^3 ) + \Delta t^4 I(u) \Big],
\end{aligned}
\end{align}
where $\phi_r^0$ is the discretization error \eqref{eqn:disc_err} at $t=t_0$.
\end{theorem}

\begin{proof}
To apply the discrete Gronwall's lemma to Lemma~\ref{lemma:uniform_bound}, we first consider the following notations:
\begin{align} \label{eqn:discrete-gronwall-equivalence-notation}
\begin{aligned}
\alpha_{n}&:= \| \phi_r^{n}\|^2_{L^2} \geq 0 ,   \\
\beta_{n}&:=  C \Delta t \Big( \|\partial \eta^{n+1}\|^2_{L^2} +  \|\eta^{n+1/2}\|^2_{L^2} + \| \eta_x^{n+1/2}\|^2_{L^2} + \Delta t^3 I_n(u) \Big)  \geq 0, \\
C &= \frac{27}{4 \, \mathcal{C} \, \nu^3} \geq 0  \quad \text{from Lemma~\ref{lemma:uniform_bound}}.
\end{aligned}
\end{align}

By using notations in~\eqref{eqn:discrete-gronwall-equivalence-notation}, 
we rewrite \eqref{eqn:EE22_CN} as follows:
\begin{align} \label{eqn:EE23_CN}
(1-C \Delta t) \alpha_{n+1}  \leq (1+C \Delta t) \alpha_n + \beta_n \, \quad \forall n=0,...,N-1.
\end{align}

By using the discrete Gronwall's lemma (see Lemma 10.4 in~\cite{rebollo2014mathematical}) in~\eqref{eqn:EE23_CN}, and if the small time step assumption, i.e., $\Delta t \le 0.5 \, C^{-1} = \frac{2 \, \mathcal{C} \, \nu^3}{27}$, is guaranteed, then the
the following inequality holds:  
\begin{align} \label{eqn:EE24_CN}
\begin{aligned}
 \max_{k=0,..,N} \| \phi_r^{k}\|^2_{L^2} & \leq e^{4CT}  \|\phi_r^0\|^2_{L^2}    \\
& + 2Ce^{4CT} \sum_{n=0}^{N-1} \Delta t \Big( \|\partial \eta^{n+1}\|^2_{L^2} + \|\eta^{n+1}\|^2_{L^2} +  \|\eta_x^{n+1}\|^2_{L^2} + \Delta t^3 I_n(u) 
\Big).
\end{aligned}
\end{align}
Now, using triangle inequality, from \eqref{eqn:EE24_CN} we get
\begin{align} \label{eqn:EE26_CN}
\begin{aligned}
 \max_{0\leq k \leq N} \|e^k\|^2_{L^2} &\leq \max_{0\leq k \leq N} \|\eta^{k}\|^2_{L^2} + C e^{4C T} \Big[ \|\phi_r^0\|^2_{L^2} + \Delta t^4 I(u) \\
& + \sum_{n=0}^{N-1} \Delta t \Big( \|\partial \eta^{n+1}\|^2_{L^2} + \|\eta^{n+1}\|^2_{L^2} +  \|\eta_x^{n+1}\|^2_{L^2} \Big) \Big].
\end{aligned}   
\end{align}

Using $ (2N+1)\Delta t = 2T + \Delta t \leq 3T $ relation and updating the generic constant $C$ in \eqref{eqn:EE26_CN} give
\begin{align} \label{eqn:EE27_CN}
\begin{aligned}
 \max_{0\leq k \leq N} \|e^k\|^2_{L^2}  \leq  C & \Big[ \max_{0\leq k \leq N} \|\eta^{k}\|^2_{L^2} + \|\phi_r^0\|^2_{L^2} + \Delta t^4 I(u)  \\
&+ \frac{1}{2N+1} \sum_{n=0}^{N-1} \Big( \|\partial \eta^{n+1}\|^2_{L^2} + \|\eta^{n+1}\|^2_{L^2} +  \|\eta_x^{n+1}\|^2_{L^2}
\Big)  \Big].
\end{aligned}
\end{align}

Now, use Lemma~\ref{lemma:ritz_proj_u_uh}, \eqref{eqn:DQ_POD_proj_error_W_norm2} in Lemma~\ref{lemma:DQ_POD_approx_errors_Wnorm} with $W=L^2$ and $H_0^1$, and \eqref{eqn:DQ_POD_uniform_bound3} in Theorem~\ref{thm:DQ_uniform_estimates} with $W=L^2$. This ends the proof.
\end{proof}

Now, we discuss the behavior of the $l^{\infty}(L^2)$ DQ-L2 error \eqref{eqn:linfinity_L2_dq_l2_error_bound}. At a first glance, \eqref{eqn:linfinity_L2_dq_l2_error_bound} does not meet any optimality type of criteria in Definition~\ref{defn:optimal} since the DQ-L2 error is built with the norm error choice $W=L^2$ and the right-hand side of \eqref{eqn:linfinity_L2_dq_l2_error_bound} is not only purely bounded with $W=L^2$ norm error. Applying \eqref{eqn:l2_pod_inv_est} in Lemma \ref{lemma:pod_inverse_estimate} to the second term, which is in $\| (\varphi_i - R_r \varphi_i)_x \|^2_{L^2}= \| \varphi_i - R_r \varphi_i \|^2_{H_0^1}$ in \eqref{eqn:linfinity_L2_dq_l2_error_bound} yields a coefficient $C^{L^2}_{inv}:=\sqrt{\|S_r\|_2}$ which depends on the ROM dimension $r$. However, when we numerically investigate the behavior of $\sqrt{\|S_r\|_2}$, we observe that it is almost constant (see the top-right plot in Figure~\ref{fig:dq_l2_convergence}). Thus, the $l^\infty(L^2)$ DQ-L2 error bound does not meet the suboptimality criteria in Remark~\ref{remark:suboptimality}; on the contrary, it holds the optimality-II type, i.e., \eqref{eqn:ROMerror_optimalII}.

Before presenting the $l^{\infty}(L^2)$ DQ-H01 error, we provide some bounds related to the Ritz projection when considering the $H_0^1(\Omega)$ POD space framework, which will be used in Theorem~\ref{theorem:linfinity_L2_dq_h01_error_bound}.

\begin{lemma} \label{lemma:DQ_ritz_proj_bound} (Bounds for Ritz Projection)
The Ritz projection satisfies the following bounds if $H_0^1(\Omega)$ POD basis is used, see Section 4.2 in \cite{locke2021new}:
\begin{eqnarray} \label{eqn:ritz_proj_bound}
\begin{cases}
\|\varphi_i - R_r \varphi_i\|_{L^2} =  \|\varphi_i\|_{L^2} \\
\|(\varphi_i - R_r \varphi_i)_x \|_{L^2} = 1,  \quad \forall i=r+1,...,d. 
\end{cases}
\end{eqnarray}
\end{lemma}

\begin{proof}
One can expand the term $R_r \varphi_i \in X^r $ by considering the first $r$ POD modes as
\begin{align} \label{eqn:ritz_proj_bound3}
 R_r \varphi_i := \sum_{j=1}^{r} ( R_r \varphi_i, \varphi_j )_{H_0^1} \varphi_j 
\end{align}

Let $W$ denote either $L^2$ or $H_0^1$. Then, by using \eqref{eqn:ritz_proj_bound3}, we get the following:
\begin{align} \label{eqn:ritz_proj_bound4}
\begin{aligned}
\|\varphi_i - R_r \varphi_i\|_{W}^2 & = \Big( \varphi_i - \sum_{j=1}^{r} \big(R_r \varphi_i, \varphi_j \big)_{H_0^1} \varphi_j , \varphi_i - \sum_{k=1}^{r} \big( R_r \varphi_i, \varphi_k \big)_{H_0^1} \varphi_k \Big)_{W} \\
& = \| \varphi_i \|^2_{W} - \sum_{j=1}^{r}  \big(R_r \varphi_i, \varphi_j \big)_{H_0^1} \big(\varphi_j, \varphi_i \big)_{W} - \sum_{k=1}^{r}  \big(R_r \varphi_i, \varphi_k \big)_{H_0^1} \big(\varphi_k, \varphi_i \big)_{W} \\
& + \sum_{j,k=1}^{r}  \big(R_r \varphi_i, \varphi_j \big)_{H_0^1} (R_r \varphi_i, \varphi_k )_{H_0^1} \big(\varphi_j, \varphi_k \big)_{W} \\
& = \| \varphi_i \|^2_{W} - \sum_{j=1}^{r}  \big( \varphi_i, \varphi_j \big)_{H_0^1} \big(\varphi_j, \varphi_i \big)_{W} - \sum_{k=1}^{r}  \big( \varphi_i, \varphi_k \big)_{H_0^1} \big(\varphi_k, \varphi_i \big)_{W} \\
& + \sum_{j,k=1}^{r}  \big( \varphi_i, \varphi_j \big)_{H_0^1} \big( \varphi_i, \varphi_k \big)_{H_0^1} \big(\varphi_j, \varphi_k \big)_{W} \\
& = \| \varphi_i \|^2_{W}
\end{aligned}
\end{align}
where $( \varphi_i, \varphi_k )_{H_0^1} = 0 , \quad \forall i=r+1,..,d, \quad j,k=1,..,r$. If $W = L^2$, then $\| \varphi_i \|^2_{W} = \| \varphi_i \|^2_{L^2} $ will keep the same; otherwise, $W = H_0^1$, then $\| \varphi_i \|^2_{H_0^1} = 1$. This ends the proof.
\end{proof}

\begin{theorem} \label{theorem:linfinity_L2_dq_h01_error_bound}
Assume that $\Delta t \le \frac{2 \, \mathcal{C} \,\nu^3}{27}$, then the $l^\infty(L^2)$ DQ-H01 error is bounded by
\begin{align} \label{eqn:linfinity_L2_dq_h01_error_bound}
\begin{aligned}
\max_{ 0 \leq k \leq N  }\|e^{k}\|^{2}_{L^2}  \leq  C \Big[ \| \phi_r^{0}\|^2_{L^2} + &  \sum_{i = r+1}^{d} 
\lambda_i^\mathrm{DQ} \big( 1 + \| \varphi_i \|^2_{L^2}  \big) + (h^{2l} + \Delta t^3 ) + \Delta t^4 I(u) \Big].
\\
\end{aligned}
\end{align}
\end{theorem}

\begin{proof}
The derivation of the $l^{\infty}(L^2)$ DQ-H01 error bound is exactly the same as the error bound for the $l^{\infty}(L^2)$ DQ-L2 error bound \eqref{eqn:linfinity_L2_dq_l2_error_bound}. Now, we consider the $H_0^1(\Omega)$ POD basis; thus, we need to bound the right-hand side of \eqref{eqn:linfinity_L2_dq_l2_error_bound} by using the properties in \eqref{eqn:ritz_proj_bound} in Lemma~\ref{lemma:DQ_ritz_proj_bound}. This ends the proof.
\end{proof}

Now, we discuss the behavior of the $l^{\infty}(L^2)$ DQ-H01 error bound \eqref{eqn:linfinity_L2_dq_h01_error_bound}. Based on optimality types in Definition~\ref{defn:optimal}, one can conclude that \eqref{eqn:linfinity_L2_dq_h01_error_bound} is optimal-I if there is no additive factor 1, which equals to $\|\varphi_i\|_{H_0^1}$ for $H_0^1$ POD. Applying \eqref{eqn:h01_pod_inv_est} in Lemma \ref{lemma:pod_inverse_estimate} to $\|\varphi_i\|_{H_0^1}$ in \eqref{eqn:linfinity_L2_dq_h01_error_bound} yields
\begin{align} \label{eqn:linfinity_L2_dq_h01_error_bound_pod_inv_est}
\begin{aligned}
1+\|\varphi_i\|^2_{L^2}& =\|\varphi_i\|^2_{H_0^1}+\|\varphi_i\|^2_{L^2}, \\
                       & \leq C_r^2~\|\varphi_i\|^2_{L^2},
\end{aligned}
\end{align}
where $C_r=\max\{1,C^{H_0^1}_{inv}\}$. We numerically investigate the behavior of $C^{H_0^1}_{inv}=\sqrt{\|M_r^{-1}\|_2}$ and find that it is almost constant (see the top-right plot in Figure~\ref{fig:dq_h01_convergence}). Thus, the $l^\infty(L^2)$ DQ-H01 error bound is optimal-I, i.e., \eqref{eqn:ROMerror_optimalI}.

\subsection{The Natural-Norm Error Estimates} \label{sec:uniform_natural_norm_error}
In this section, we provide the natural-norm, i.e., ($l^\infty(L^2) \cap l^2(H_0^1)$) error estimates for the DQ ROM~\eqref{eqn:CN_POD_G_ROM} considering both $L^2(\Omega)$ and $H_0^1(\Omega)$ POD spaces. Furthermore, we discuss the behavior of the natural-norm DQ-L2 and DQ-H01 error in 
Theorem~\ref{theorem:natural_norm_dq_l2_error_bound} and Theorem~\ref{theorem:natural_norm_dq_h01_error_bound}, respectively.

\begin{theorem} \label{theorem:natural_norm_dq_l2_error_bound}
The natural-norm DQ-L2 error is bounded by 
\begin{align} \label{eqn:natural_norm_dq_l2_error_bound}
\begin{aligned}
\max_{0\leq k \leq N} \|e^{k}\|^{2}_{L^2} + \nu \Delta t \sum_{n=0}^{N-1} \| e_x^{n+1/2} \|^2_{L^2} \leq C & \Big[\| \phi_r^{0}\|^2_{L^2} +   \sum_{i = r+1}^{d} \lambda_i^\mathrm{DQ} \Big( \| \varphi_i - R_r \varphi_i \|^2_{L^2} \\
&+ \| (\varphi_i - R_r \varphi_i)_x \|^2_{L^2} \Big) + (h^{2l} + \Delta t^3 )  + \Delta t^4 I(u) \Big].
\end{aligned}
\end{align}
\end{theorem}

\begin{proof}
We start derivation with rearranging \eqref{eqn:EE18_CN} as follows:
\begin{align} \label{eqn:EE28a_CN}
\begin{aligned}
\big(  \| \phi_r^{n+1}\|^2_{L^2} &- \| \phi_r^{n}\|^2_{L^2} \big) + \nu \Delta t \| (\phi_r)_x^{n+1/2}\|^2_{L^2} \leq C \Delta t \Big[\|\phi_r^{n+1}\|^2_{L^2} + \|\phi_r^{n}\|^2_{L^2}  \\ 
& + \|\partial \eta^{n+1}\|^2_{L^2} + \|\eta^{n+1/2}\|^2_{L^2} +  \|\eta_x^{n+1/2}\|^2_{L^2} + \Delta t^3 I_n(u) \Big],
\end{aligned}
\end{align}
where $C=\frac{27}{4 \, \mathcal{C} \, \nu^3}$. The first two terms on the right-hand side of \eqref{eqn:EE28a_CN} are in ROM space; thus, they are bounded thanks to the standard stability estimate for the ROM solution. 
Then, summing from $n=0$ to $n=N-1$ and using triangle inequality give 
\begin{align} \label{eqn:EE28_CN}
\begin{aligned}
\max_{0\leq k \leq N}  \|e^k\|^2_{L^2} + \nu \Delta t \sum_{n=0}^{N-1} \|e_x^{n+1/2} \|^2_{L^2}  \leq C \Big[ \max_{0\leq k \leq N} \|\eta^{k}\|^2_{L^2} +  \|\phi_r^0\|^2_{L^2}  \\
 + \Delta t \sum_{n=0}^{N-1} \Big( \|\partial \eta^{n+1}\|^2_{L^2} + \|\eta^{n+1/2}\|^2_{L^2} + \|\eta_x^{n+1/2}\|^2_{L^2}
\Big)  + \Delta t^4 I(u) \Big],
\end{aligned}   
\end{align}
Use $ (2N+1)\Delta t = 2T + \Delta t \leq
3T $ relation and update the generic constant $C$ in \eqref{eqn:EE28_CN}, then apply Lemma~\ref{lemma:ritz_proj_u_uh}, \eqref{eqn:DQ_POD_proj_error_W_norm2} in Lemma~\ref{lemma:DQ_POD_approx_errors_Wnorm} with $W=L^2$ and $H_0^1$, and \eqref{eqn:DQ_POD_uniform_bound3} in Theorem~\ref{thm:DQ_uniform_estimates} with $W=L^2$. This ends the proof.
\end{proof}

Now, we discuss the behavior of the natural-norm DQ-L2 error bound \eqref{eqn:natural_norm_dq_l2_error_bound}. If $l^{\infty}(L^2)$ and $l^{2}(H_0^1)$ parts of the natural-norm in \eqref{eqn:natural_norm_dq_l2_error_bound} are controlled by $\| \varphi_i - R_r \varphi_i \|^2_{L^2}$ and $\| (\varphi_i - R_r \varphi_i)_x \|^2_{L^2}$, respectively, then based on Definition~\ref{defn:optimal}, one can conclude that the natural-norm DQ-L2 \eqref{eqn:natural_norm_dq_l2_error_bound} is optimal-II.

\begin{theorem} \label{theorem:natural_norm_dq_h01_error_bound}
The natural-norm DQ-H01 error is bounded by 
\begin{align}  \label{eqn:natural_norm_dq_h01_error_bound}
\begin{aligned}
\max_{0\leq k \leq N}  \|e^k\|^2_{L^2}  + \nu \Delta t \sum_{n=0}^{N-1} \|e_x^{n+1/2}\|^2_{L^2} & \leq C \Big[\| \phi_r^{0}\|^2_{L^2} + \sum_{i = r+1}^{d} \lambda_i^\mathrm{DQ} \big( 1+ \| \varphi_i  \|^2_{L^2} \big) \\
& + (h^{2l} + \Delta t^3 ) + \Delta t^4 I(u) \Big].
\end{aligned}
\end{align}
\end{theorem}

\begin{proof}
The derivation of the natural-norm DQ-H01 error bound is exactly the same as the error bound for the natural-norm DQ-L2 error bound \eqref{eqn:natural_norm_dq_l2_error_bound}. Now, we consider the $H_0^1(\Omega)$ POD basis; thus, we need to bound the right-hand side of \eqref{eqn:natural_norm_dq_l2_error_bound} by using the properties in \eqref{eqn:ritz_proj_bound} in Lemma~\ref{lemma:DQ_ritz_proj_bound}. This ends the proof.
\end{proof}

Now, we discuss the behavior of the natural-norm DQ-H01 error bound \eqref{eqn:natural_norm_dq_h01_error_bound}. If $l^{\infty}(L^2)$ and $l^{2}(H_0^1)$ parts of the natural-norm in \eqref{eqn:natural_norm_dq_h01_error_bound} are controlled by $\|\varphi_i \|^2_{L^2}$ and $1=\| (\varphi_i)_x\|^2_{L^2}$, respectively, then based on Definition~\ref{defn:optimal}, one can conclude that the natural-norm DQ-H01 \eqref{eqn:natural_norm_dq_h01_error_bound} is optimal-I.

\begin{remark} \label{remark:nodq_max_natural_norm_error_l2_h01_pod}
We briefly provide the $l^{\infty}(L^2)$ and natural-norm noDQ-L2 and noDQ-H01 error estimates. To obtain the $l^{\infty}(L^2)$ and natural-norm noDQ error bounds, one can proceed similarly to the above proof using the $ L^2 $ projection instead of the Ritz projection. The $l^{\infty}(L^2)$ and natural-norm error bounds are the same. Specifically, the $l^{\infty}(L^2)$ and natural-norm noDQ-L2 error bound is provided in  \eqref{eqn:nodq_l2_error_bound}, and the noDQ-H01 error bound is provided in \eqref{eqn:nodq_h01_error_bound}. 
\begin{subequations} 
\begin{align}
\mathcal{E}  & \leq C \Big[\| \phi_r^{0}\|^2_{L^2} +  \sum_{i = r+1}^{d} \lambda_i^\mathrm{noDQ} \Big( \| \varphi_i \|^2_{L^2} + \| (\varphi_i)_x \|^2_{L^2} \Big)  + (h^{2l} + \Delta t^3 ) + \Delta t^4 I(u) \Big], \label{eqn:nodq_l2_error_bound} \\
\mathcal{E}  & \leq C \Big[\| \phi_r^{0}\|^2_{L^2} +  \sum_{i = r+1}^{d} \lambda_i^\mathrm{noDQ} \big( 1 + \| \varphi_i \|^2_{L^2} \big)  + (h^{2l} + \Delta t^3 ) + \Delta t^4 I(u) \Big]. \label{eqn:nodq_h01_error_bound}
\end{align}
\end{subequations}

Now, we discuss the behavior of the noDQ-L2 \eqref{eqn:nodq_l2_error_bound} and noDQ-H01 \eqref{eqn:nodq_h01_error_bound} error bounds by applying the same process as we did in the DQ case. 

For $l^{\infty}(L^2)$ noDQ-L2 error bound, applying \eqref{eqn:l2_pod_inv_est} in Lemma \ref{lemma:pod_inverse_estimate} to the second term, which is in $\| (\varphi_i)_x \|^2_{L^2}$, in \eqref{eqn:nodq_l2_error_bound} yields a coefficient $C^{L^2}_{inv}:=\sqrt{\|S_r\|_2}$ which depends on the ROM dimension $r$. We numerically investigate the behavior of $\sqrt{\|S_r\|_2}$, and observe that it increases as r increases (see the top-right plot in Figure~\ref{fig:nodq_l2_convergence}). Since the $l^\infty(L^2)$ noDQ-L2 error bound does meet the suboptimality criteria in Remark~\ref{remark:suboptimality}, $l^\infty(L^2)$ noDQ-L2 error bound \eqref{eqn:nodq_l2_error_bound} is suboptimal.

For natural-norm noDQ-L2 error bound, if $l^{\infty}(L^2)$ and $l^{2}(H_0^1)$ parts of the natural-norm in \eqref{eqn:nodq_l2_error_bound} are controlled by $\| \varphi_i \|^2_{L^2}$ and $\| (\varphi_i)_x \|^2_{L^2}$, respectively, then based on Definition~\ref{defn:optimal}, one can conclude that the natural-norm noDQ-L2 \eqref{eqn:nodq_l2_error_bound} is optimal-I.

For $l^{\infty}(L^2)$ noDQ-H01 error bound, applying \eqref{eqn:h01_pod_inv_est} in Lemma \ref{lemma:pod_inverse_estimate} to $\|\varphi_i\|_{H_0^1}$ in \eqref{eqn:nodq_h01_error_bound} yields the coefficient $C^{H_0^1}_{inv}=\sqrt{\|M_r^{-1}\|_2}$. We numerically investigate the behavior of $C^{H_0^1}_{inv}$ and observe that it increases as r increases (see the top-right plot in Figure~\ref{fig:nodq_h01_convergence}). Since, the $l^\infty(L^2)$ noDQ-H01 error bound does meet the suboptimality criteria in Remark~\ref{remark:suboptimality}, $l^\infty(L^2)$ noDQ-H01 error bound \eqref{eqn:nodq_h01_error_bound} is suboptimal.

For natural-norm noDQ-H01 error bound, if $l^{\infty}(L^2)$ and $l^{2}(H_0^1)$ parts of the natural-norm in \eqref{eqn:nodq_h01_error_bound} are controlled by $\| \varphi_i \|^2_{L^2}$ and $1=\| (\varphi_i)_x \|^2_{L^2}$, respectively, then based on Definition~\ref{defn:optimal}, one can conclude that the natural-norm noDQ-H01 \eqref{eqn:nodq_l2_error_bound} is optimal-I.
\end{remark}

The behavior of the error bounds, which are theoretically derived in Sections \ref{sec:uniform_max_l2_error} and \ref{sec:uniform_natural_norm_error}, are summarized in Table~\ref{table:theoretical_convergence_rate}. Considering $L^2$ and $H_0^1$ POD basis, error norms, and noDQ/DQ frameworks, we observe that for the DQ errors bounds with all cases and the natural norm noDQ-L2 and noDQ-H01 are optimal; whereas the $l^{\infty}(L^2)$ noDQ-L2 and noDQ-H01 error bounds are suboptimal.

\begin{table}[h!] 
\centering
\begin{tabular}{| c| c| c| } 
\hline
& $l^{\infty}(L^2)$ & \text{natural norm}  \\ 
\hline 
 & suboptimal & optimal-I \\ 
noDQ-L2 & Section~\ref{sec:uniform_natural_norm_error} & Section~\ref{sec:uniform_natural_norm_error} \\
& \eqref{eqn:nodq_l2_error_bound} in Remark~\ref{remark:nodq_max_natural_norm_error_l2_h01_pod} & \eqref{eqn:nodq_l2_error_bound} in Remark~\ref{remark:nodq_max_natural_norm_error_l2_h01_pod} \\ 
\hline 
 & suboptimal & optimal-I \\ 
noDQ-H01 & Section~\ref{sec:uniform_natural_norm_error} & Section~\ref{sec:uniform_natural_norm_error} \\
& \eqref{eqn:nodq_h01_error_bound} in Remark~\ref{remark:nodq_max_natural_norm_error_l2_h01_pod} & 
\eqref{eqn:nodq_h01_error_bound} in Remark~\ref{remark:nodq_max_natural_norm_error_l2_h01_pod}
\\ 
\hline 
 & optimal-II & optimal-II \\ 
DQ-L2 & Section~\ref{sec:uniform_max_l2_error} & Section~\ref{sec:uniform_natural_norm_error} \\
& \eqref{eqn:linfinity_L2_dq_l2_error_bound} in Theorem~\ref{theorem:linfinity_L2_dq_l2_error_bound} &   \eqref{eqn:natural_norm_dq_l2_error_bound} in Theorem~\ref{theorem:natural_norm_dq_l2_error_bound} 
\\ 
\hline 
 & optimal-I & optimal-I \\ 
DQ-H01 & Section~\ref{sec:uniform_max_l2_error} & Section~\ref{sec:uniform_natural_norm_error} \\
& \eqref{eqn:linfinity_L2_dq_h01_error_bound} in 
Theorem~\ref{theorem:linfinity_L2_dq_h01_error_bound} & \eqref{eqn:natural_norm_dq_h01_error_bound} in Theorem~\ref{theorem:natural_norm_dq_h01_error_bound} \\
& \eqref{eqn:ROMerror_truly_optimal} in Definition~\ref{defn:optimal} & 
\eqref{eqn:ROMerror_truly_optimal} in Definition~\ref{defn:optimal}
\\ 
\hline 
\end{tabular}
\caption{Theoretical results: The behavior of the noDQ and DQ ROMs with $L^2(\Omega)$ and $H_0^1(\Omega)$ POD basis and $l^{\infty}(L^2)$ and natural norm errors.}  \label{table:theoretical_convergence_rate}
\end{table}

\section{Numerical Results} \label{sec:numerical_results}
In this section, we provide numerical results for the Burgers equation~\eqref{eqn:burgers} with the following initial condition
\begin{equation} \label{eqn:step_IC}
u_0(x)=\begin{cases}
\displaystyle~ 1, & x \in (0,1/2], \\
~\displaystyle 0, & x \in (1/2,1]. 
\end{cases}
\end{equation}
This condition generates a smooth solution at any $t>0$, with an infinite time gradient at $t=0$. This allows us to test the role of the difference quotients. To obtain full order model (FOM) data (the FOM simulation is provided in Figure~\ref{fig:fom_simulation}), we solve \eqref{eqn:burgers} by using the finite element method considering $\nu=10^{-2}$, $f = 0$, mesh size $h=1/512$, piecewise linear finite elements for the spatial discretization, and Crank-Nicolson time discretization. A small time step $\Delta t = 10^{-3}$ is taken to obtain the errors due to the POD discretization.

\begin{figure}[h!] 
\begin{center}
\includegraphics[width=0.45\textwidth,height=0.3\textwidth]{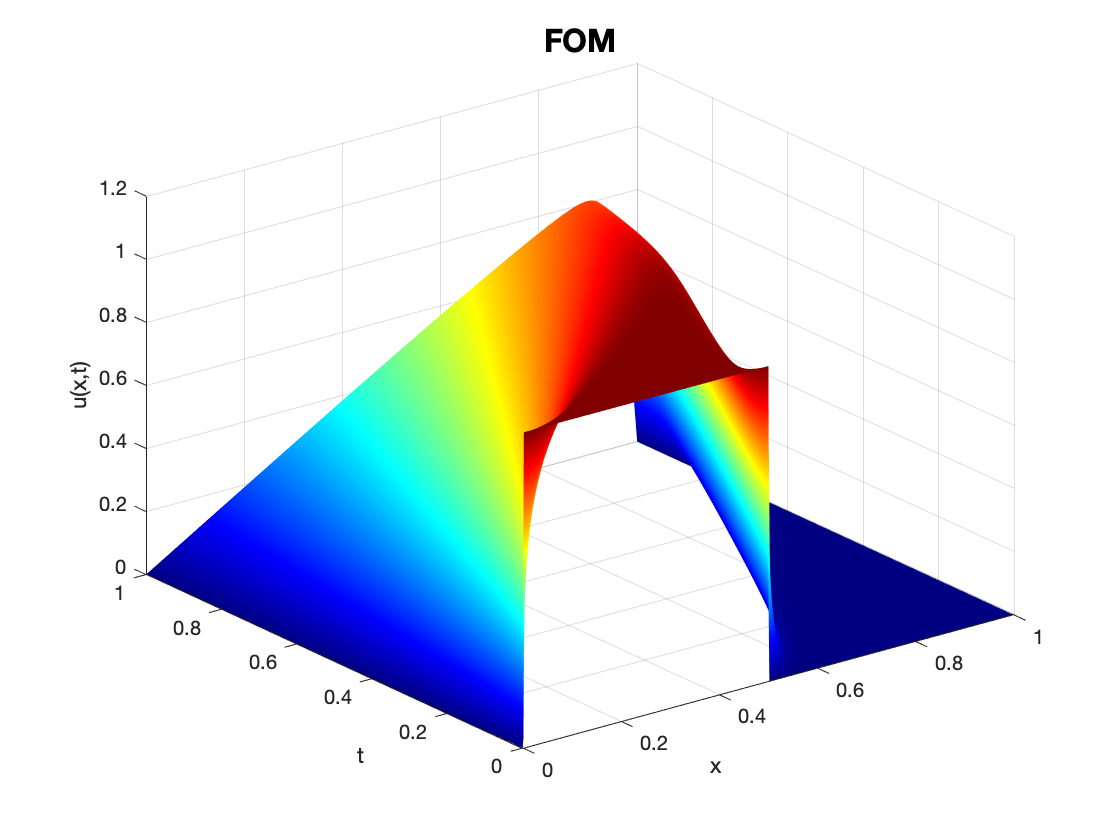}
\caption{FOM simulation for Burgers equation~\eqref{eqn:burgers} with initial condition~\eqref{eqn:step_IC}.}
\label{fig:fom_simulation}
\end{center} 
\end{figure}

For all test cases, we compute two different absolute norm ROM errors:
\begin{subequations}
\begin{align} \label{eqn:max_L2_error}
\mathcal{E}_{l^\infty(L^2)} &= \max_{0 \leq k \leq N} \|e^k\|^2_{L^2}, \\
\mathcal{E}_{l^\infty(L^2) \cap l^2(H_0^1)} &= \max_{0 \leq k \leq N} \|e^k\|^2_{L^2} + \nu \Delta t \sum_{n=0}^{N-1}\| e_x^{n+1/2}\|^2_{L^2},  \label{eqn:natural_norm_error}
\end{align}
\end{subequations}
being $e=u_h-u_r$.

\begin{remark}
Since the ROM initial condition is chosen as $u_r^0:=R_ru^0$, the discretization error \eqref{eqn:linfinity_L2_dq_l2_error_bound} at $t=t_0$, i.e., $\phi_r^0=0$ in all the error bounds derivation in numerical results.
\end{remark}

\subsection{noDQ ROM Results} \label{sec:nodq_rom_numerics}
In this section, we numerically discuss the behavior of the $l^{\infty}(L^2)$ and natural-norm noDQ-L2 and noDQ-H01 error bounds. Based on the noDQ error estimates in \eqref{eqn:nodq_l2_error_bound}-\eqref{eqn:nodq_h01_error_bound} in Remark~\ref{remark:nodq_max_natural_norm_error_l2_h01_pod}, we define the following RHS terms:
\begin{subequations}
\begin{align}
\text{noDQ-RHS1} &= \sum_{i = r+1}^{d} \lambda_i^\mathrm{noDQ} \Big( \| \varphi_i \|^2_{L^2} + \| (\varphi_i)_x \|^2_{L^2} \Big) + \Delta t^2  + \Delta t^4 I(u), \label{eqn:nodq_rhs1} \\
\text{noDQ-RHS2} &= \sum_{i = r+1}^{d} \lambda_i^\mathrm{noDQ} (1+ \| \varphi_i\|^2_{L^2})  + \Delta t^2 + \Delta t^4 I(u). \label{eqn:nodq_rhs2}
\end{align}
\end{subequations}

\begin{figure}[h!] 
\begin{center}
\includegraphics[width=0.45\textwidth,height=0.3\textwidth]{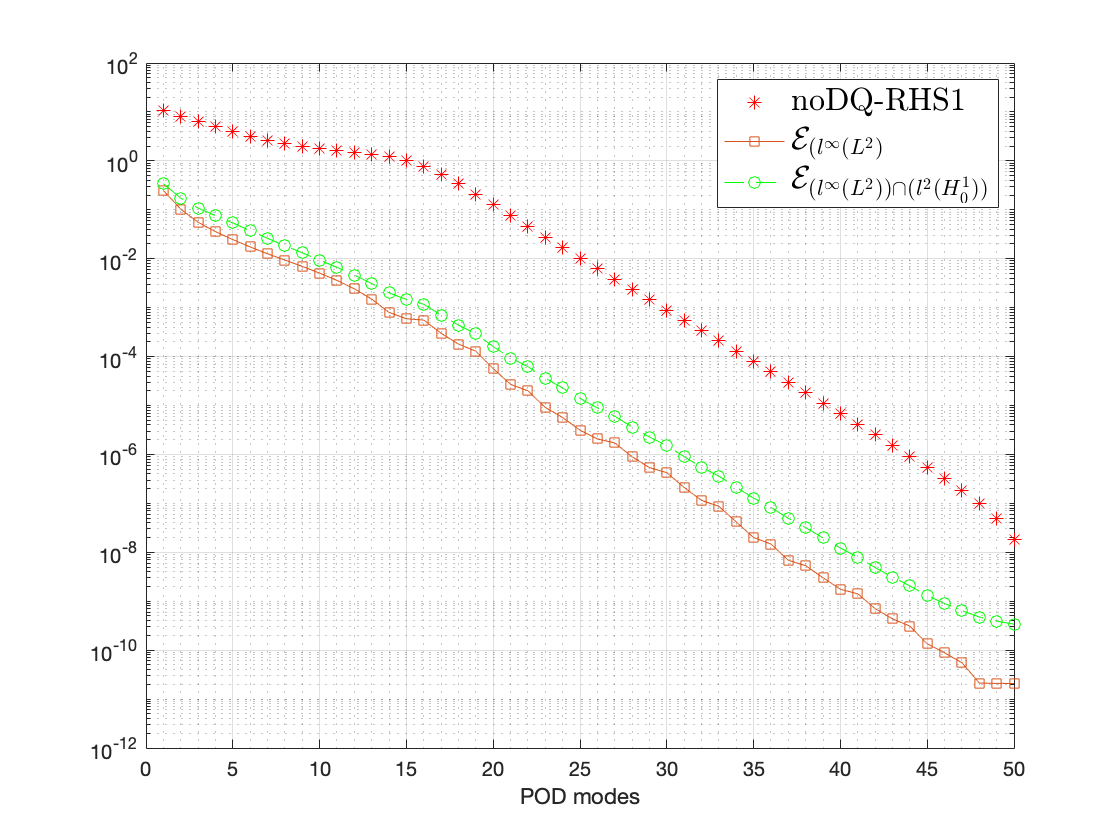}
\includegraphics[width=0.45\textwidth,height=0.3\textwidth]{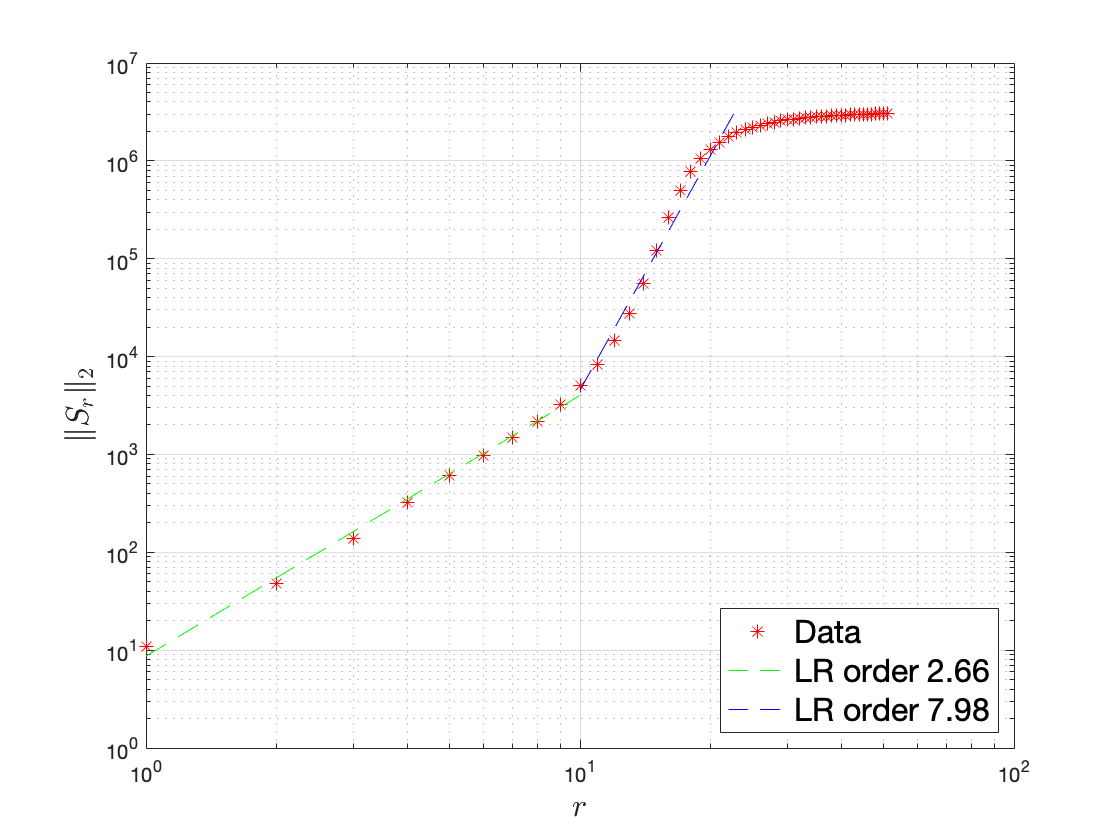}
\includegraphics[width=0.45\textwidth,height=0.3\textwidth]{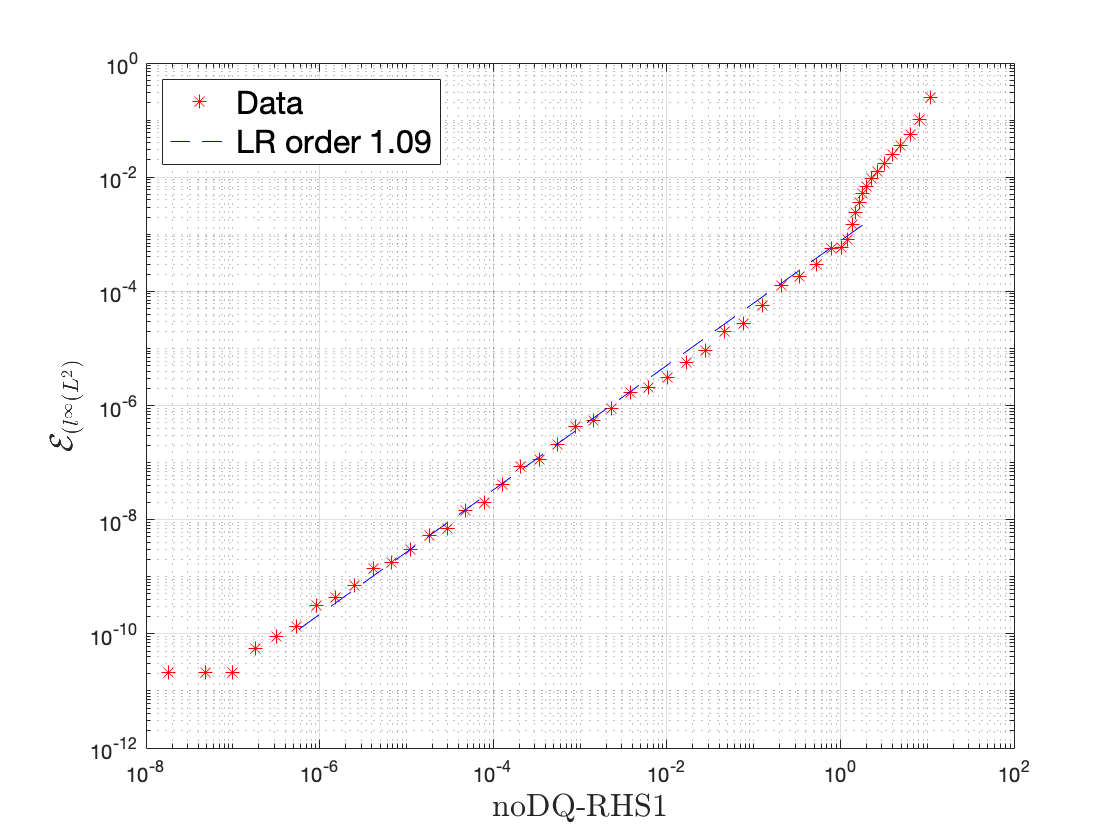}
\includegraphics[width=0.45\textwidth,height=0.3\textwidth]{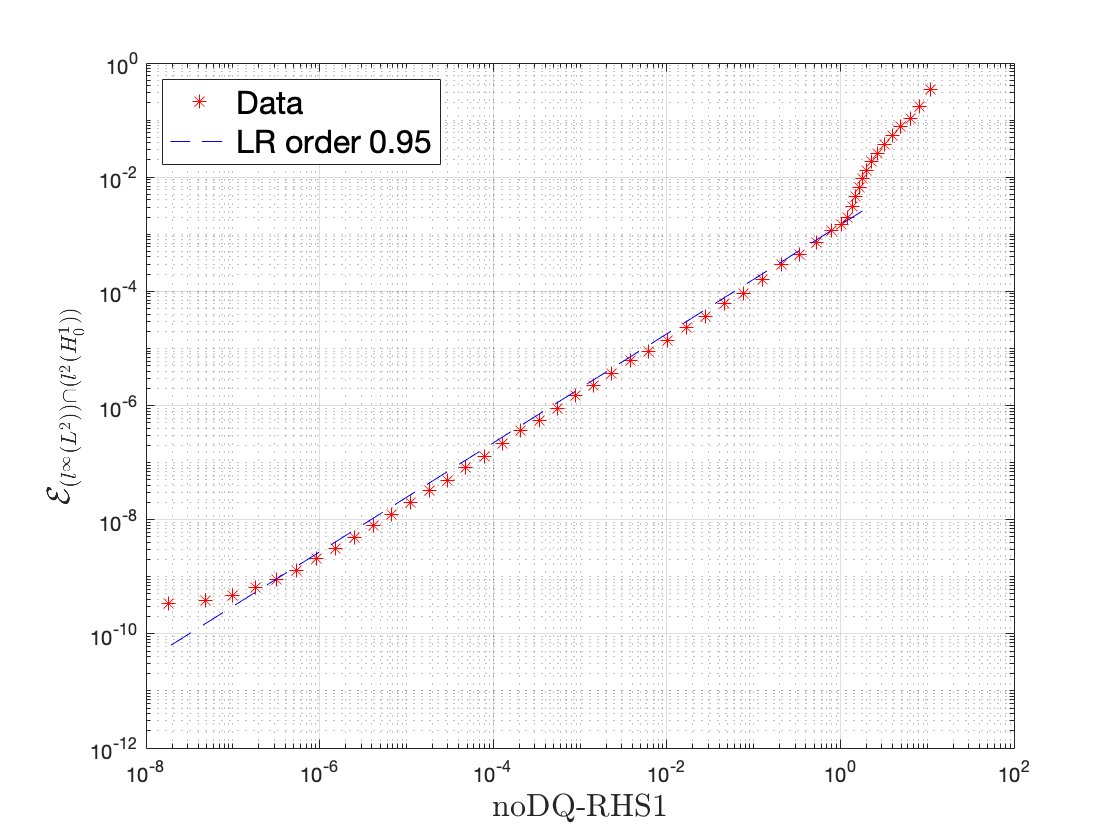}
\caption{The behavior of the $l^{\infty}(L^2)$ and natural-norm noDQ-L2 errors.
} \label{fig:nodq_l2_convergence}
 \end{center} 
\end{figure}

In the top left plot of Figure~\ref{fig:nodq_l2_convergence}, we plot the noDQ-RHS1 defined in \eqref{eqn:nodq_rhs1} and $l^{\infty}(L^2)$, natural-norm noDQ-L2 errors in \eqref{eqn:nodq_l2_error_bound}. We observe that the $l^{\infty}(L^2)$ and natural-norm noDQ-L2 errors stay below the noDQ-RHS1. The top right plot shows how the scaling of $\|S_r\|_2$, which is defined above Lemma~\ref{lemma:pod_inverse_estimate}, changes as $r$ changes. For the bottom plots in Figure~\ref{fig:nodq_l2_convergence}, we plot the linear regression (LR) orders for $l^{\infty}(L^2)$ and natural-norm noDQ-L2 errors, from left to right, respectively. Since the LR orders in the bottom plots are close to 1, we numerically observe that $l^{\infty}(L^2)$ and natural-norm noDQ-L2 errors are optimal. However, theoretical discussions in Remark~\ref{remark:nodq_max_natural_norm_error_l2_h01_pod} show that the $l^{\infty}(L^2)$  and natural-norm noDQ-L2 error bounds are suboptimal and optimal, respectively. 

In the top left plot of Figure~\ref{fig:nodq_h01_convergence}, we plot the noDQ-RHS2 defined in \eqref{eqn:nodq_rhs2} and $l^{\infty}(L^2)$, natural-norm noDQ-H01 errors in \eqref{eqn:nodq_h01_error_bound}. The top right plot shows how the scaling of $\|M_r^{-1}\|_2$, which is defined above Lemma~\ref{lemma:pod_inverse_estimate}, changes as $r$ changes. For the bottom plots in Figure~\ref{fig:nodq_h01_convergence}, we plot the linear regression (LR) orders for $l^{\infty}(L^2)$ and natural-norm noDQ-H01 errors, from left to right, respectively. Since the LR orders in the bottom plots are close to 1, we conclude that $l^{\infty}(L^2)$ and natural-norm noDQ-H01 errors are optimal. However, theoretical discussions in Remark~\ref{remark:nodq_max_natural_norm_error_l2_h01_pod} show that the $l^{\infty}(L^2)$  and natural-norm noDQ-H01 error bounds are suboptimal and optimal, respectively. 

In Figure~\ref{fig:nodq_rom_soln}, we plot the noDQ-L2 and noDQ-H01 solutions with two different $r$ values, i.e., $r=5, 20$. For both $r$ values, we observe that the noDQ-H01 yields slightly more accurate results than the noDQ-L2, especially for low $r$ values such as $r=5$.

\begin{figure}[h!] 
\begin{center}
\includegraphics[width=0.45\textwidth,height=0.3\textwidth]{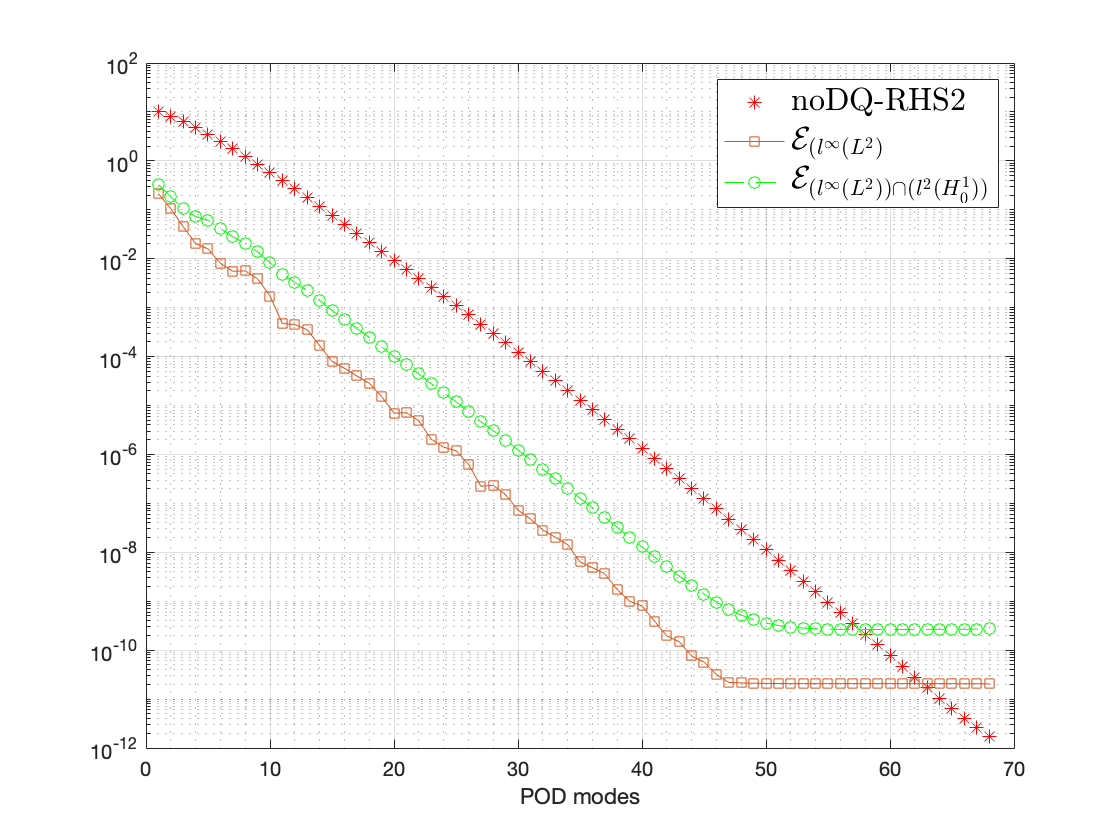}
\includegraphics[width=0.45\textwidth,height=0.3\textwidth]{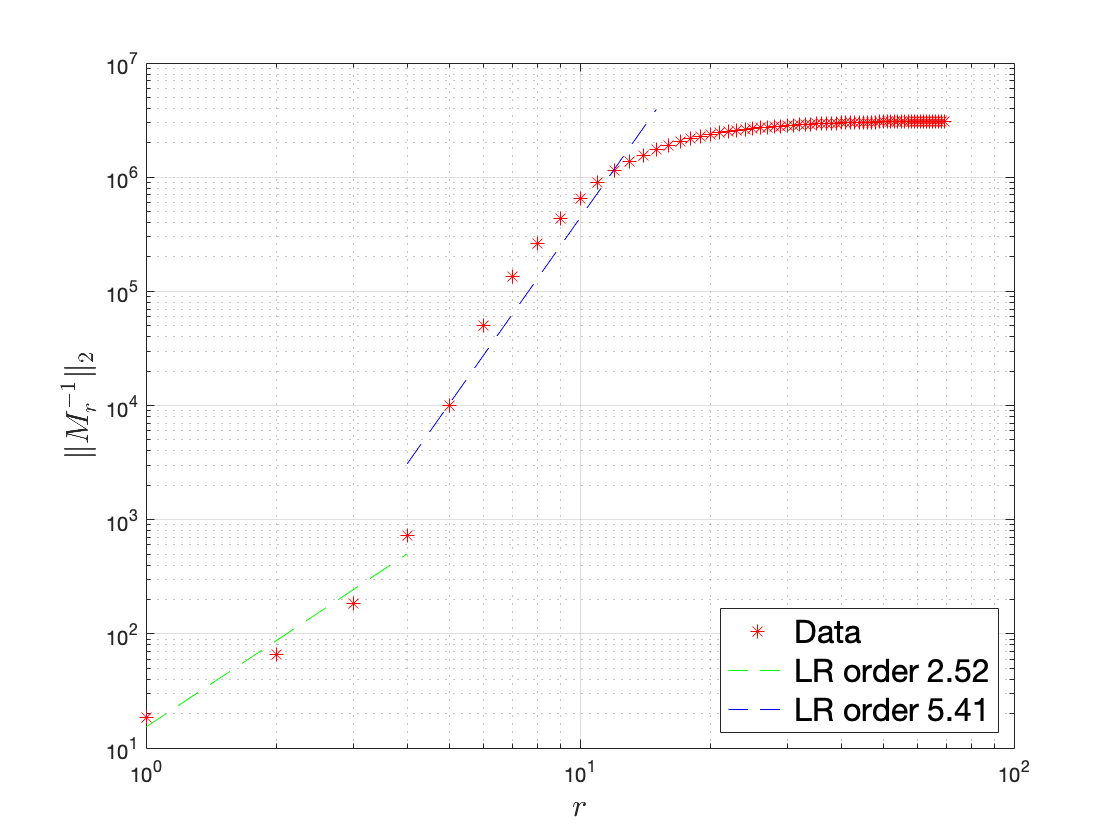}
\includegraphics[width=0.45\textwidth,height=0.3\textwidth]{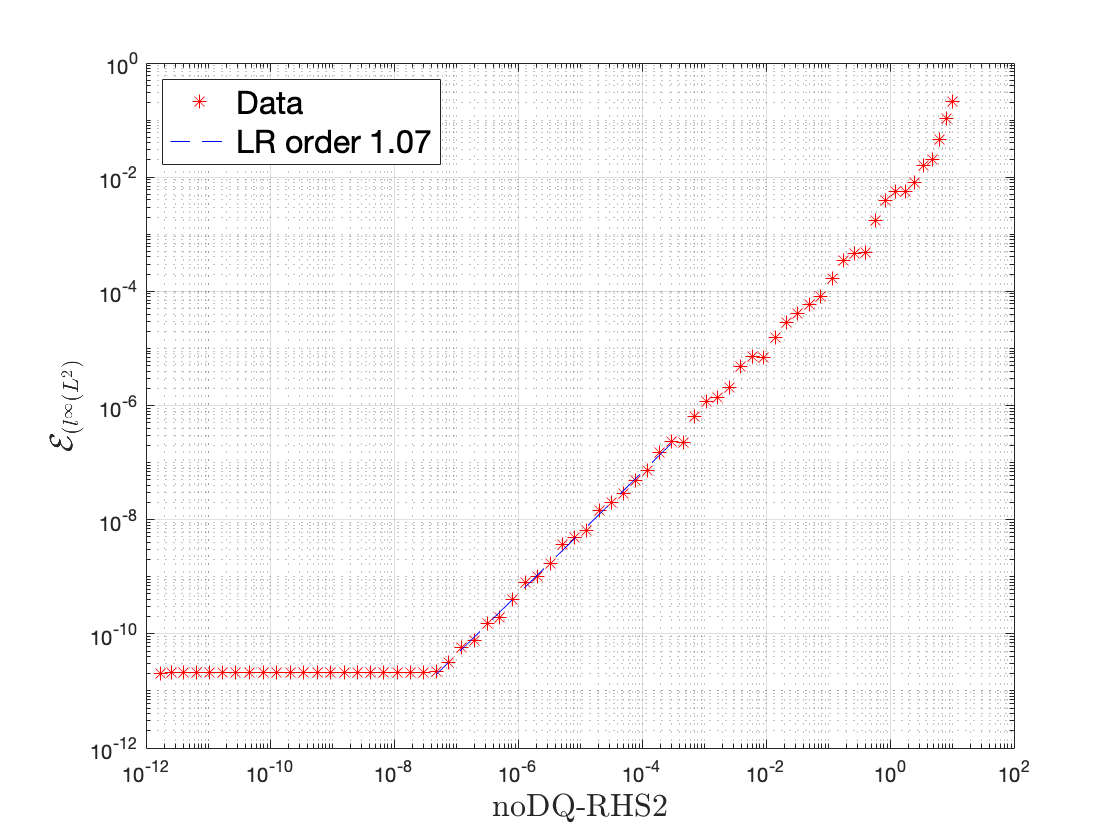}
\includegraphics[width=0.45\textwidth,height=0.3\textwidth]{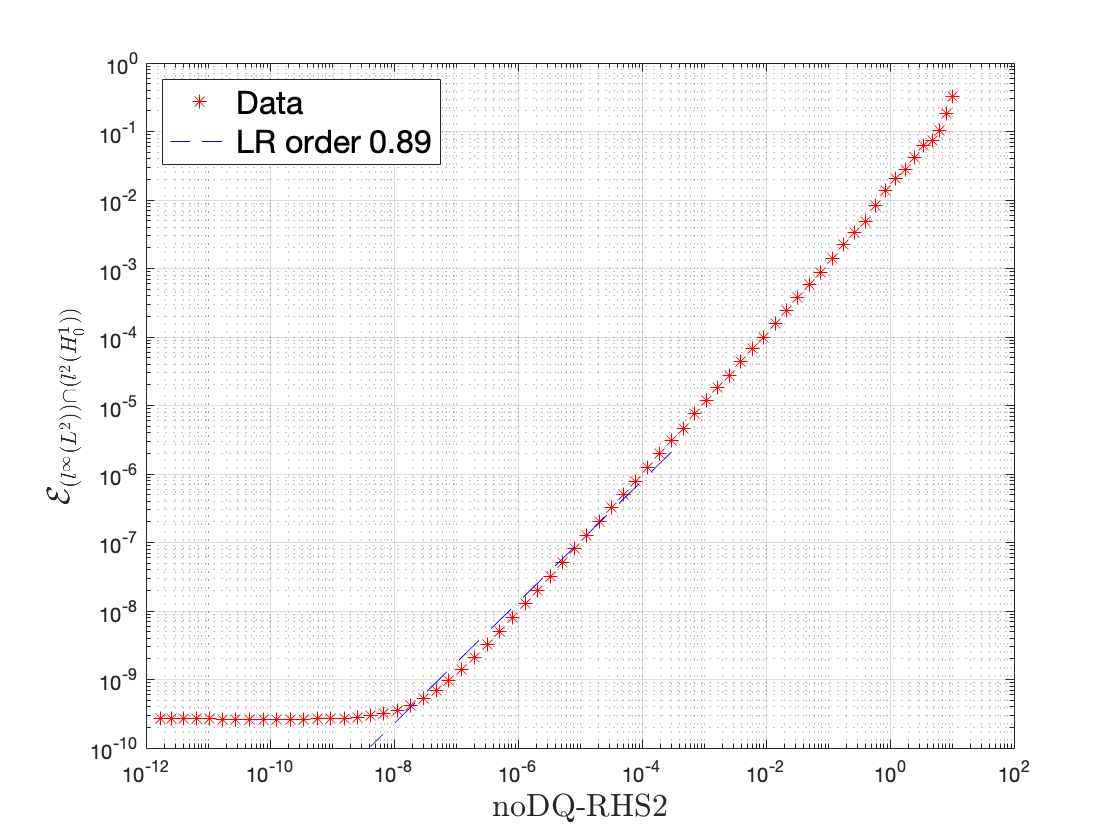}
\caption{The behavior of the $l^{\infty}(L^2)$ and natural-norm noDQ-H01 errors.
} \label{fig:nodq_h01_convergence}
 \end{center} 
\end{figure}

\begin{figure}[h!] 
\begin{center}
\includegraphics[width=0.45\textwidth,height=0.3\textwidth]{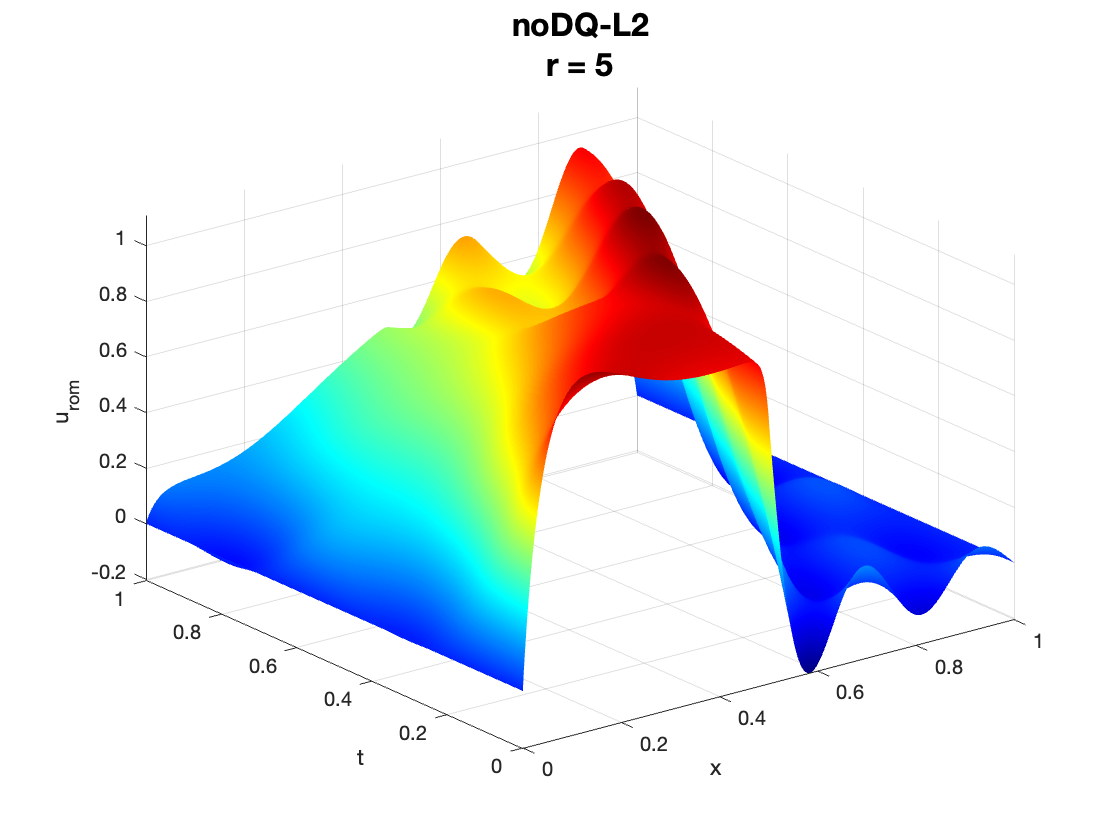}
\includegraphics[width=0.45\textwidth,height=0.3\textwidth]{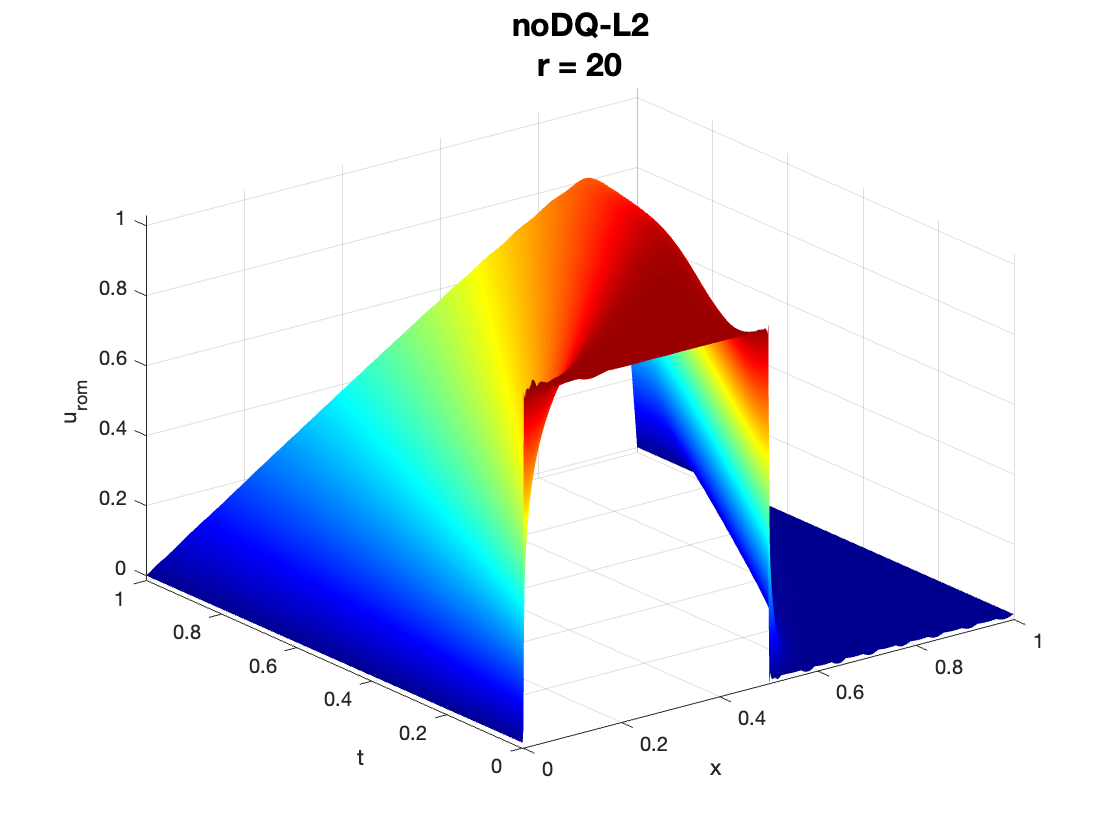}
\includegraphics[width=0.45\textwidth,height=0.3\textwidth]{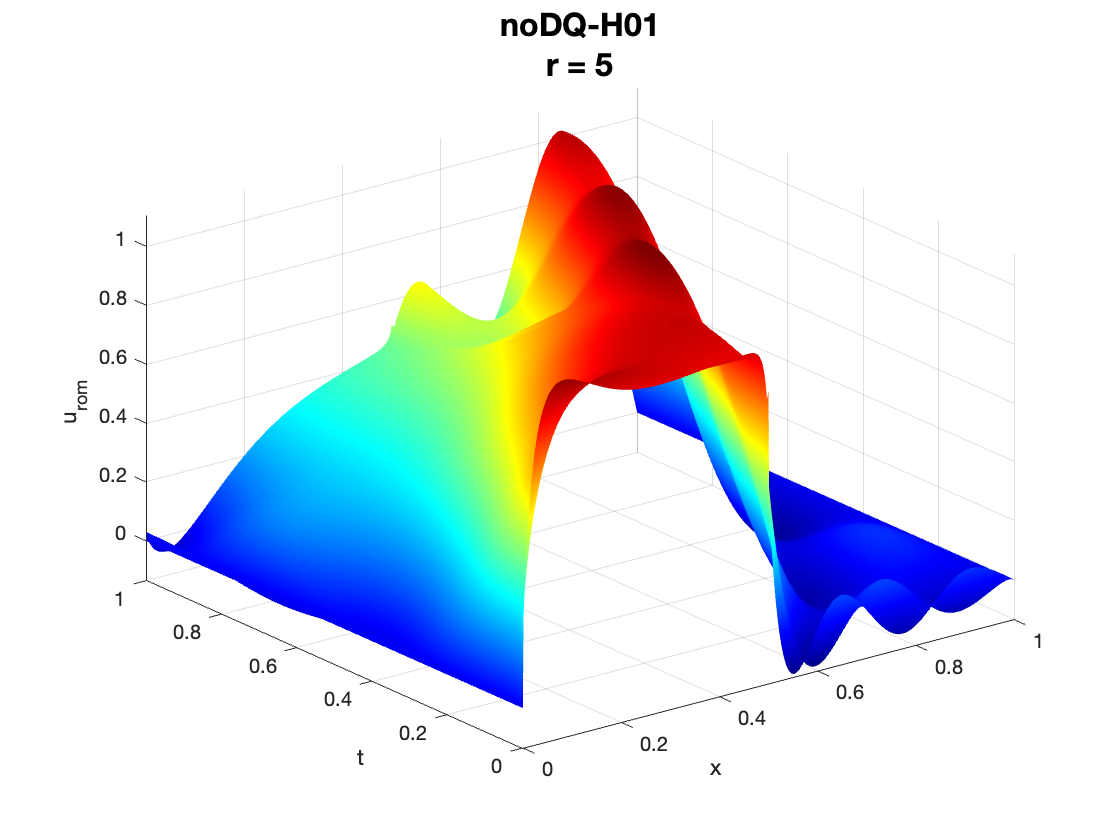}
\includegraphics[width=0.45\textwidth,height=0.3\textwidth]{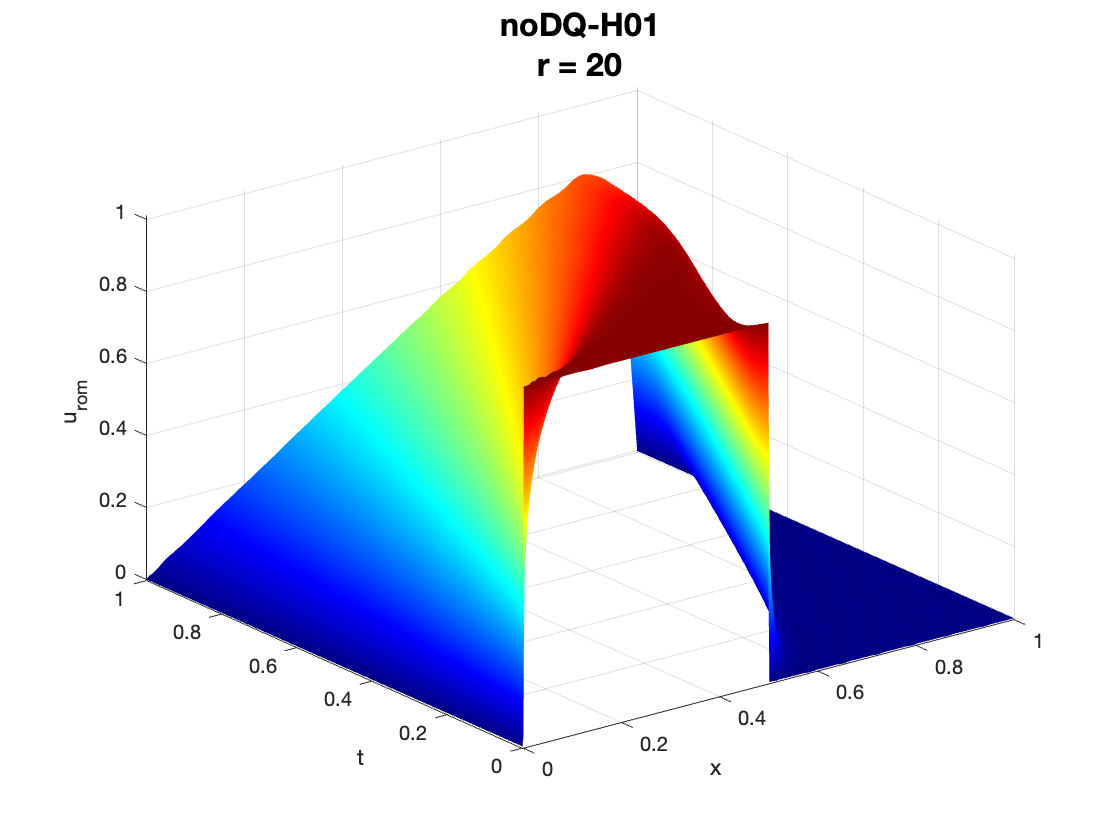}
\caption{Comparison of the noDQ-L2 and noDQ-H01 solution plots with two different $r$ values. 
} \label{fig:nodq_rom_soln}
\end{center} 
\end{figure}

\subsection{DQ ROM Results}
\label{sec:dq_rom_numerics}
In this section, we numerically discuss the behavior of the $l^{\infty}(L^2)$ and natural-norm DQ ROM errors considering $L^2(\Omega)$ and $H_0^1(\Omega)$ POD bases.

Based on the DQ-L2 error estimates in \eqref{eqn:linfinity_L2_dq_l2_error_bound} and \eqref{eqn:natural_norm_dq_l2_error_bound}, we define the following RHS term:
\begin{subequations}
\begin{align}
\text{DQ-RHS1} &= \sum_{i = r+1}^{d} \lambda_i^\mathrm{DQ} \Big( \| \varphi_i - R_r \varphi_i \|^2_{L^2} + \| (\varphi_i - R_r \varphi_i)_x \|^2_{L^2} \Big) + \Delta t^2 + \Delta t^4 I(u), \label{eqn:dq_rhs1}
\end{align}
\end{subequations}
to discuss the behavior of the DQ-L2 error estimates in \eqref{eqn:linfinity_L2_dq_l2_error_bound} and \eqref{eqn:natural_norm_dq_l2_error_bound}. 
In the top left plot of Figure~\ref{fig:dq_l2_convergence}, we plot the DQ-RHS1 defined in \eqref{eqn:dq_rhs1} and $l^{\infty}(L^2)$, natural-norm noDQ-L2 errors in \eqref{eqn:linfinity_L2_dq_l2_error_bound} and \eqref{eqn:natural_norm_dq_l2_error_bound}. We observe that the $l^{\infty}(L^2)$ and natural-norm DQ-L2 errors stay below the DQ-RHS1. The top right plot shows how the scaling of $\|S_r\|_2$, which is defined above Lemma~\ref{lemma:pod_inverse_estimate}, changes as $r$ changes. 
For the bottom plots in Figure~\ref{fig:dq_l2_convergence}, we plot the linear regression (LR) orders for $l^{\infty}(L^2)$ and natural-norm DQ-L2 errors, from left to right, respectively. 
The LR order for the $l^{\infty}(L^2)$ DQ-L2 error bound is more than 1.5 (1 is considered optimal); thus, we conclude that the $l^{\infty}(L^2)$ DQ-L2 is superoptimal; whereas we theoretically prove that it is optimal (see the discussion after Theorem~\ref{theorem:linfinity_L2_dq_l2_error_bound}).
For the natural-norm DQ-L2 error bound, the LR order is around 1; thus, we conclude that the natural-norm DQ-L2 is optimal. The behavior of the natural-norm error bound theoretically and numerically match (see the discussion after Theorem~\ref{theorem:natural_norm_dq_l2_error_bound} for the theoretical conclusion). 

\begin{figure}[h!] 
\begin{center}
\includegraphics[width=0.45\textwidth,height=0.3\textwidth]{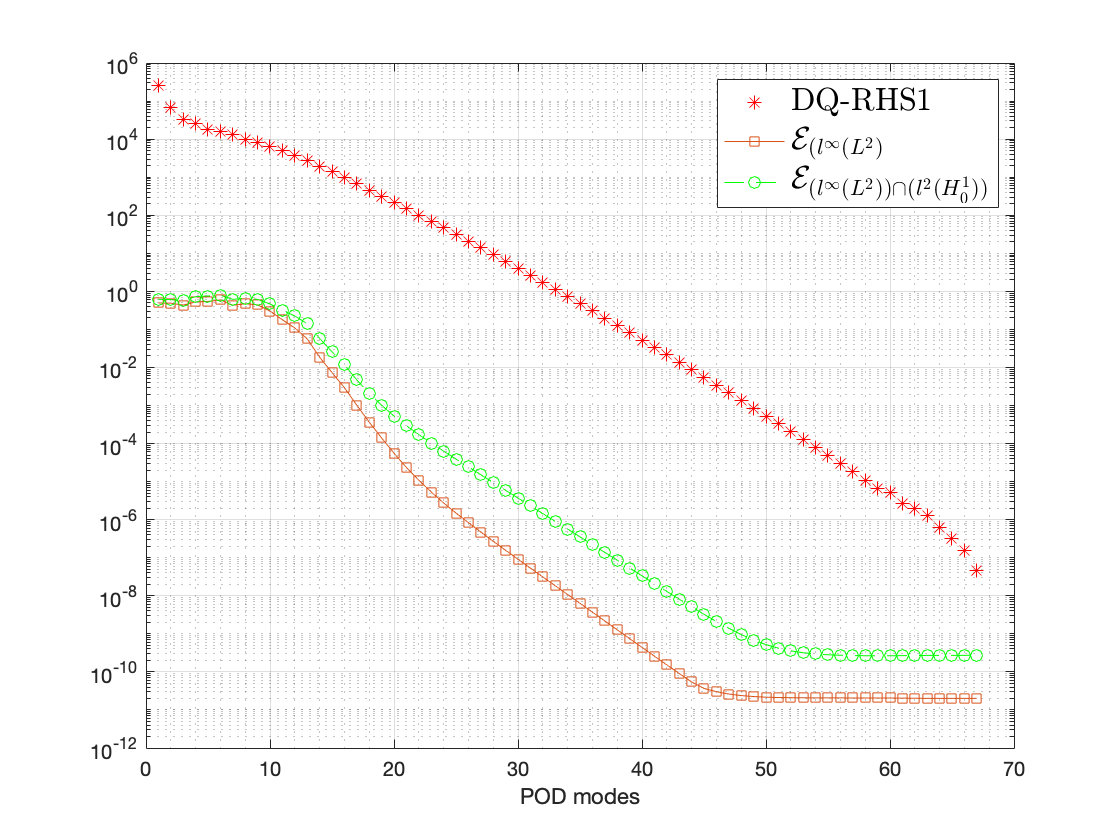}
\includegraphics[width=0.45\textwidth,height=0.3\textwidth]{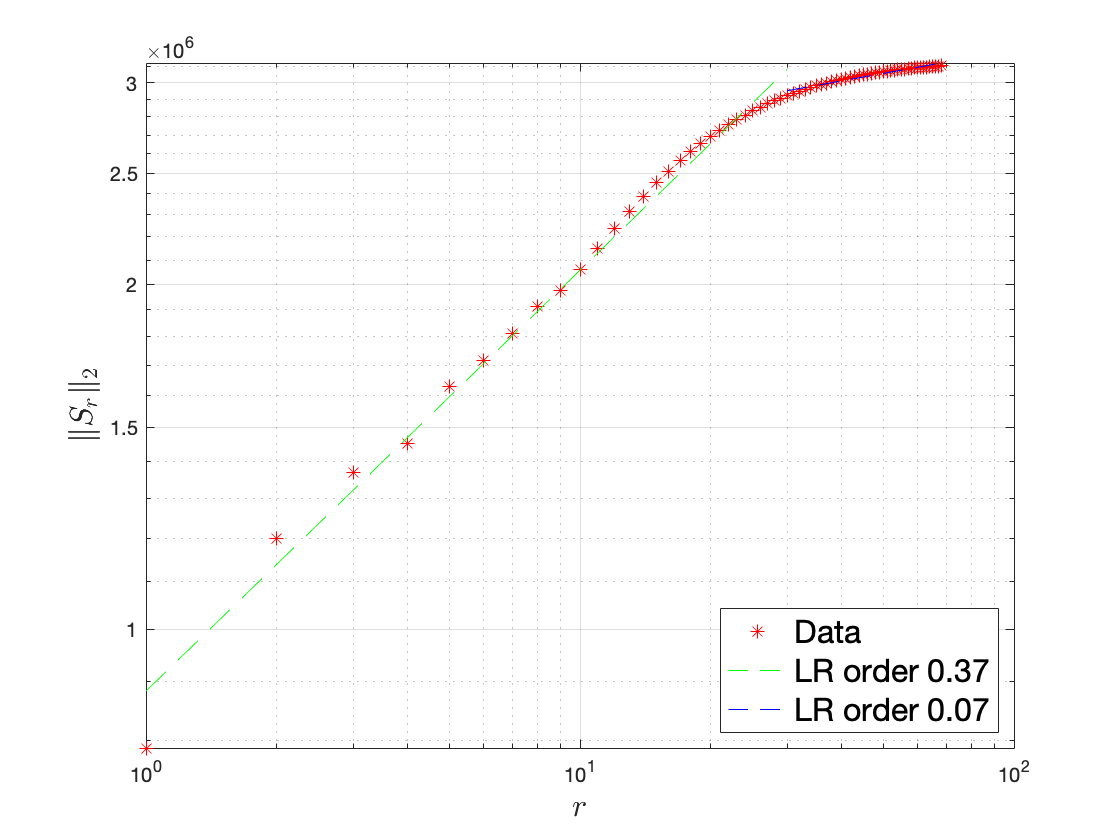}
\includegraphics[width=0.45\textwidth,height=0.3\textwidth]{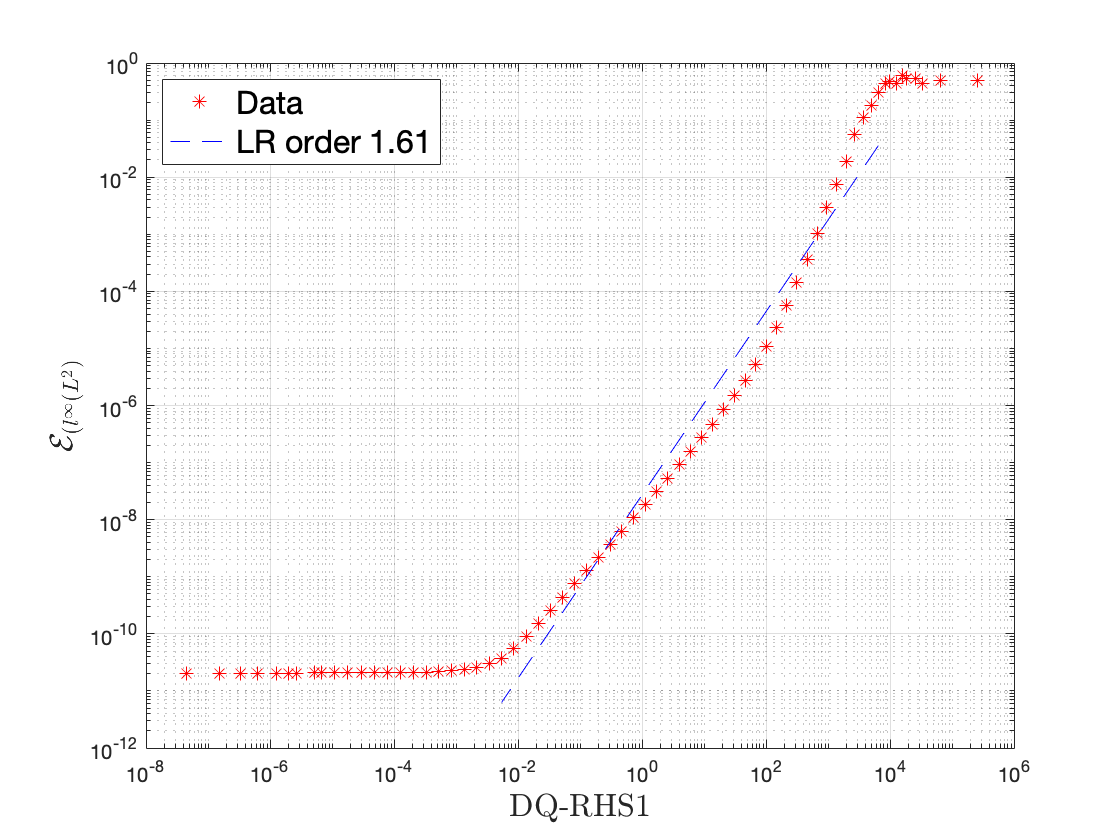}
\includegraphics[width=0.45\textwidth,height=0.3\textwidth]{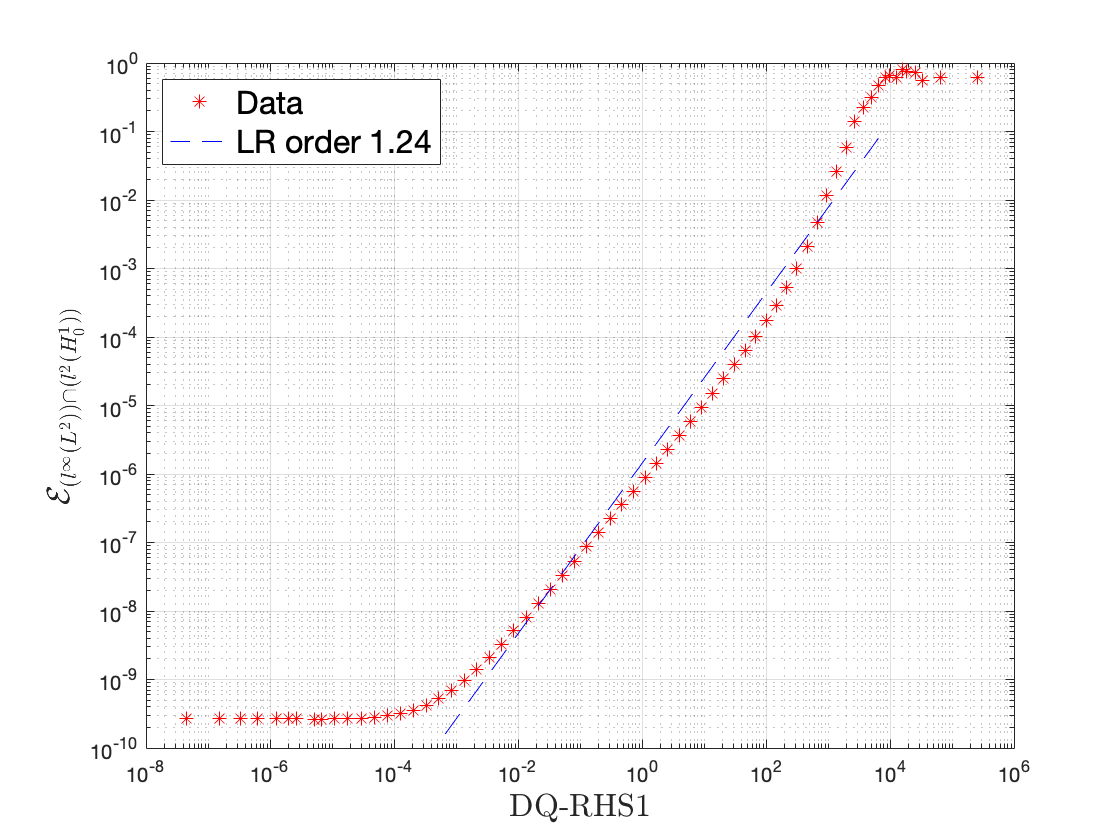}
\caption{The behavior of the $l^{\infty}(L^2)$ and natural-norm DQ-L2 errors.
} \label{fig:dq_l2_convergence}
\end{center} 
\end{figure}

Based on the DQ-H01 error estimates in \eqref{eqn:linfinity_L2_dq_h01_error_bound} and \eqref{eqn:natural_norm_dq_h01_error_bound}, we define the following RHS term:
\begin{subequations}
\begin{align}
\text{DQ-RHS2} &= \sum_{i = r+1}^{d} \lambda_i^\mathrm{DQ} \big( 1 + \| \varphi_i \|^2_{L^2}  \big) + \Delta t^2 + \Delta t^4 I(u), \label{eqn:dq_rhs2} 
\end{align}
\end{subequations}
to discuss the behavior of the DQ-H01 error estimates in \eqref{eqn:linfinity_L2_dq_h01_error_bound} and \eqref{eqn:natural_norm_dq_h01_error_bound}. 
In the top left plot of Figure~\ref{fig:dq_h01_convergence}, we plot the DQ-RHS2 defined in \eqref{eqn:dq_rhs2} and $l^{\infty}(L^2)$, natural-norm noDQ-H01 errors in \eqref{eqn:linfinity_L2_dq_h01_error_bound} and \eqref{eqn:natural_norm_dq_h01_error_bound}. We observe that the $l^{\infty}(L^2)$ and natural-norm DQ-H01 errors stay below the DQ-RHS2. Furthermore, a sudden decrease in the $l^{\infty}(L^2)$, and natural-norm DQ-H01 errors arises when a large enough number of POD modes is attempted. We think this may be related to the sharp gradients of the solution that are well represented on the DQ ROM basis only when its dimension is large enough. The top right plot shows how the scaling of $\|M_r^{-1}\|_2$, which is defined above Lemma~\ref{lemma:pod_inverse_estimate}, changes as $r$ changes. 
For the bottom plots in Figure~\ref{fig:dq_h01_convergence}, we plot the linear regression (LR) orders for $l^{\infty}(L^2)$ and natural-norm DQ-H01 errors, from left to right, respectively. 
The LR order for the $l^{\infty}(L^2)$ DQ-H01 error bound is more than 1.5 (1 is considered optimal); thus, we conclude that the $l^{\infty}(L^2)$ DQ-H01 is superoptimal; whereas we theoretically prove that is is optimal (see the discussion after Theorem~\ref{theorem:linfinity_L2_dq_h01_error_bound}). 
For the natural-norm DQ-H01 error bound, the LR order is around 1; thus, we conclude that the natural-norm DQ-H01 is optimal. The behavior of the natural-norm error bound theoretically and numerically match (see the discussion after Theorem~\ref{theorem:natural_norm_dq_h01_error_bound} for the theoretical conclusion). 

\begin{figure}[h!] 
\begin{center}
\includegraphics[width=0.45\textwidth,height=0.3\textwidth]{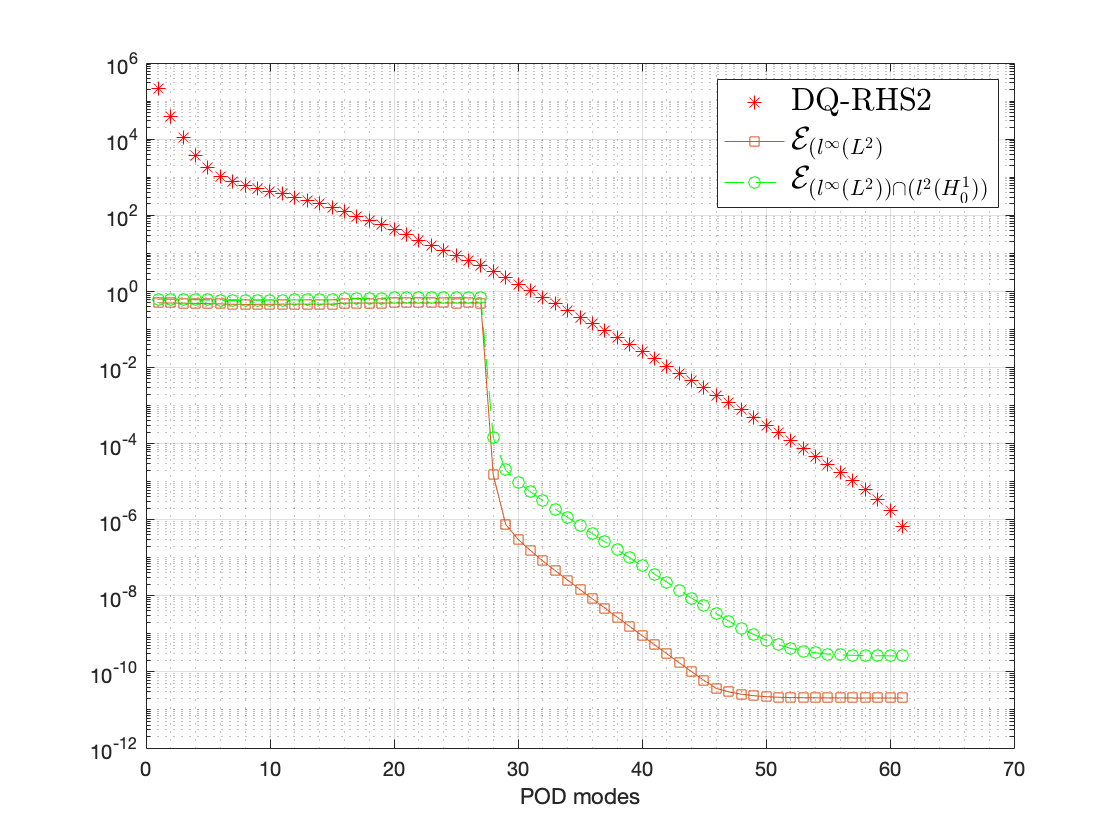}
\includegraphics[width=0.45\textwidth,height=0.3\textwidth]{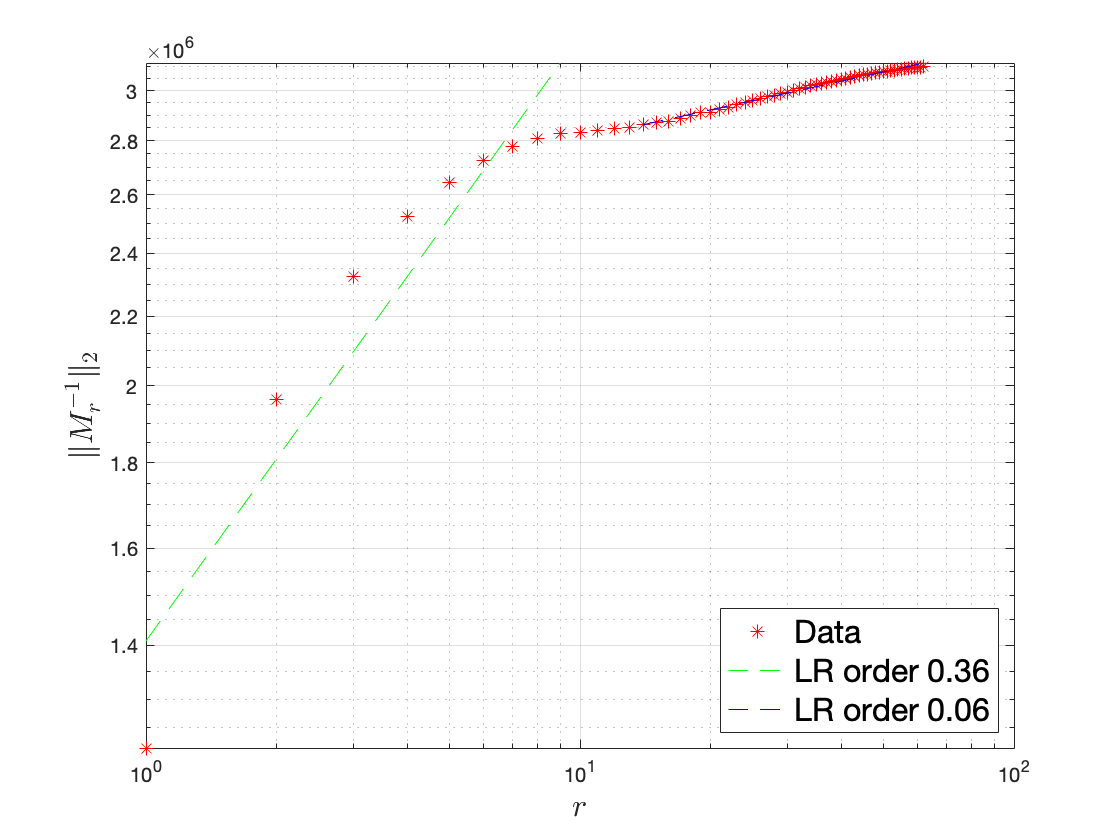}
\includegraphics[width=0.45\textwidth,height=0.3\textwidth]{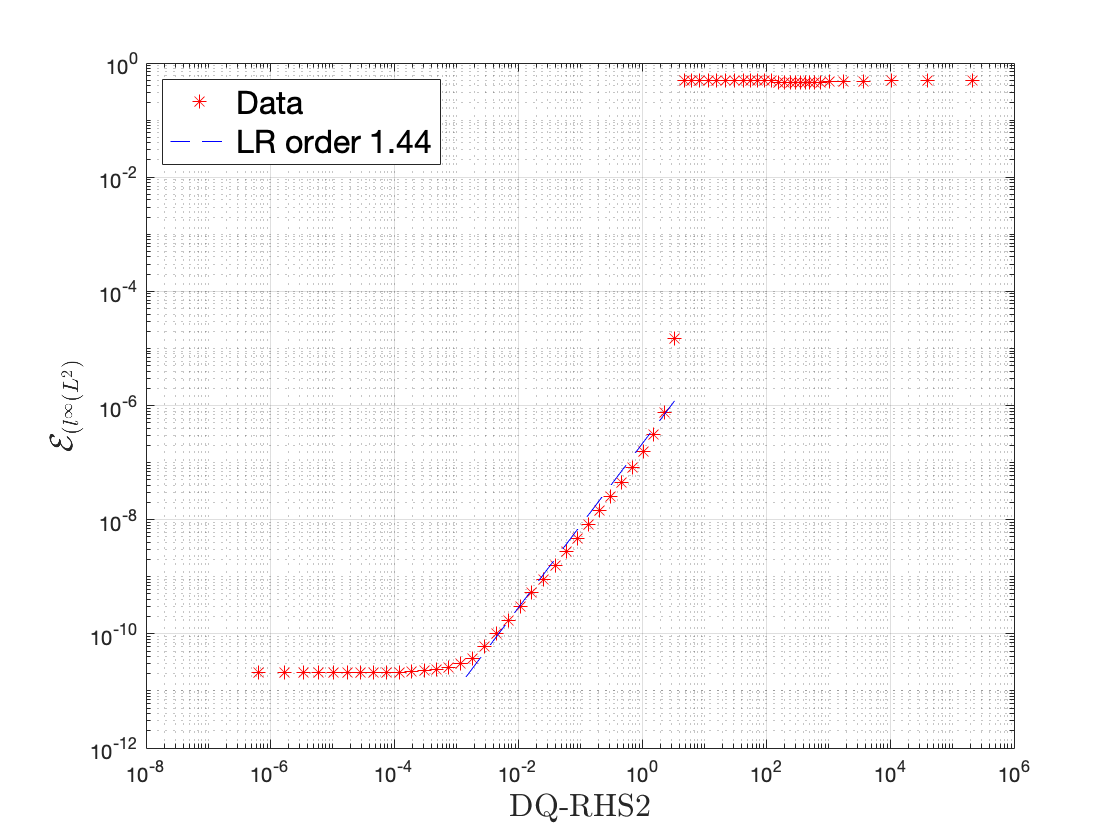}
\includegraphics[width=0.45\textwidth,height=0.3\textwidth]{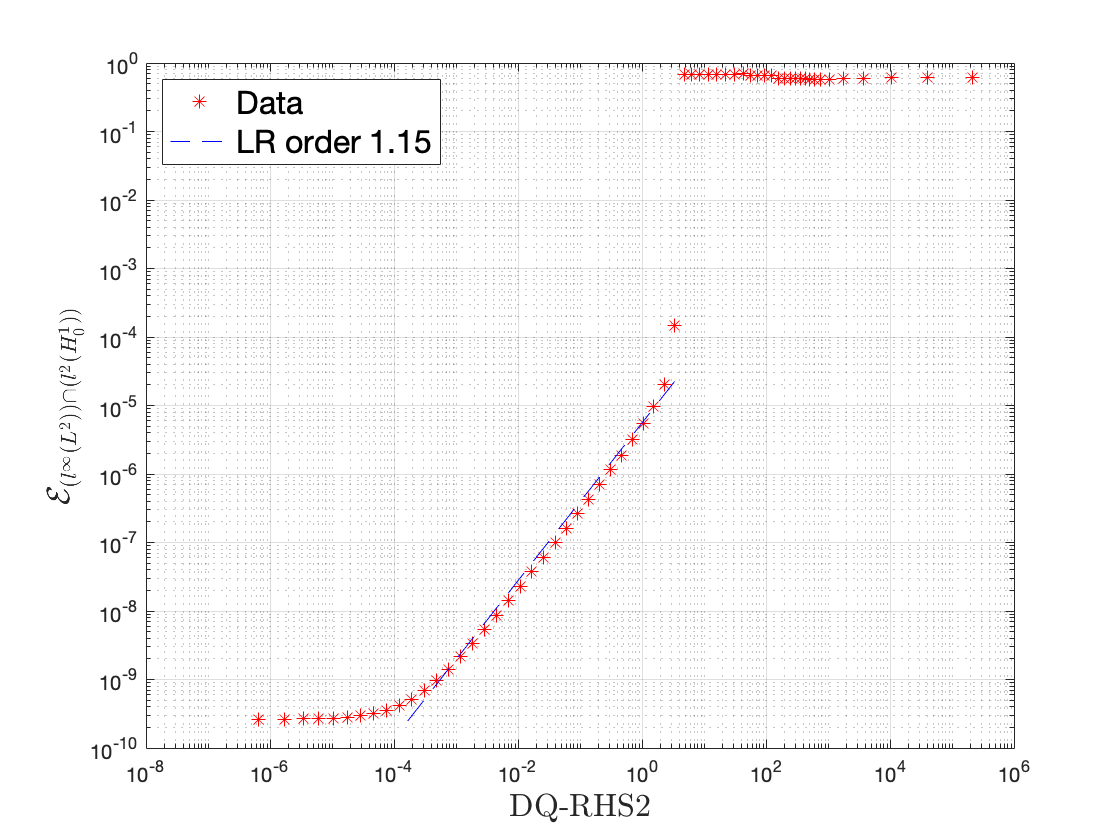}
\caption{The behavior of the $l^{\infty}(L^2)$ and natural-norm DQ-H01 errors.
} \label{fig:dq_h01_convergence}
\end{center} 
\end{figure}

\begin{table}[h!] 
\centering
\begin{tabular}{| c| c| c| } 
\hline
& $l^{\infty}(L^2)$ & \text{natural norm}  \\ 
\hline 
 & optimal & optimal \\ 
noDQ-L2 & Section~\ref{sec:nodq_rom_numerics} & Section~\ref{sec:nodq_rom_numerics} \\
& bottom-left plots in Figure \ref{fig:nodq_l2_convergence} & bottom-right plots in Figure \ref{fig:nodq_l2_convergence} \\ 
\hline 
 & optimal & optimal \\ 
noDQ-H01 & Section~\ref{sec:nodq_rom_numerics} & Section~\ref{sec:nodq_rom_numerics} \\
& bottom-left plots in Figure \ref{fig:nodq_h01_convergence} & bottom-right plots in Figure \ref{fig:nodq_h01_convergence}
\\ 
\hline 
 & superoptimal & optimal \\ 
DQ-L2 & Section~\ref{sec:dq_rom_numerics} & Section~\ref{sec:dq_rom_numerics} \\
& bottom-left plots in Figure \ref{fig:dq_l2_convergence} &   bottom-right plots in Figure \ref{fig:dq_l2_convergence}
\\ 
\hline 
 & superoptimal & optimal \\ 
DQ-H01 & Section~\ref{sec:dq_rom_numerics} & Section~\ref{sec:dq_rom_numerics} \\
& bottom-left plots in Figure \ref{fig:dq_h01_convergence} &   bottom-right plots in Figure \ref{fig:dq_h01_convergence}
\\ 
\hline 
\end{tabular}
\caption{Numerical results: The behavior of the noDQ and DQ ROMs with $L^2(\Omega)$ and $H_0^1(\Omega)$ POD basis and $l^{\infty}(L^2)$ and natural norm errors.}  \label{table:numerical_convergence_rate}
\end{table} 

In Figure~\ref{fig:dq_rom_soln}, we plot the DQ-L2 and DQ-H01 solutions with different $r$ values, i.e., $r=13,28$. For $r=13$, the DQ-H01 solution is less accurate than the DQ-L2 one since the ROM dimension $r$ does not exceed the threshold value, which is $r=27$ as can be observed in Figure~\ref{fig:dq_h01_convergence}. Furthermore, the plots show that when enough POD modes are guaranteed, the DQ-H01 solution rapidly yields an accurate solution. 

\begin{figure}[h!] 
\begin{center}
\includegraphics[width=0.45\textwidth,height=0.3\textwidth]{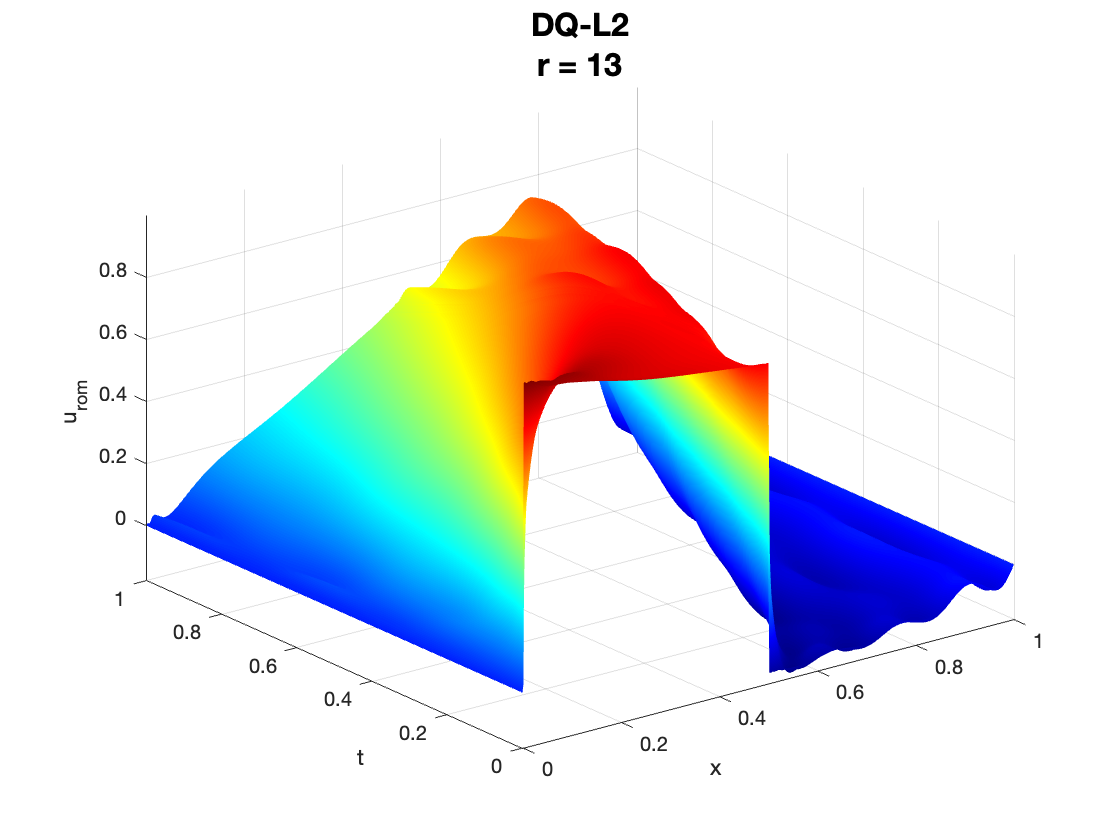}
\includegraphics[width=0.45\textwidth,height=0.3\textwidth]{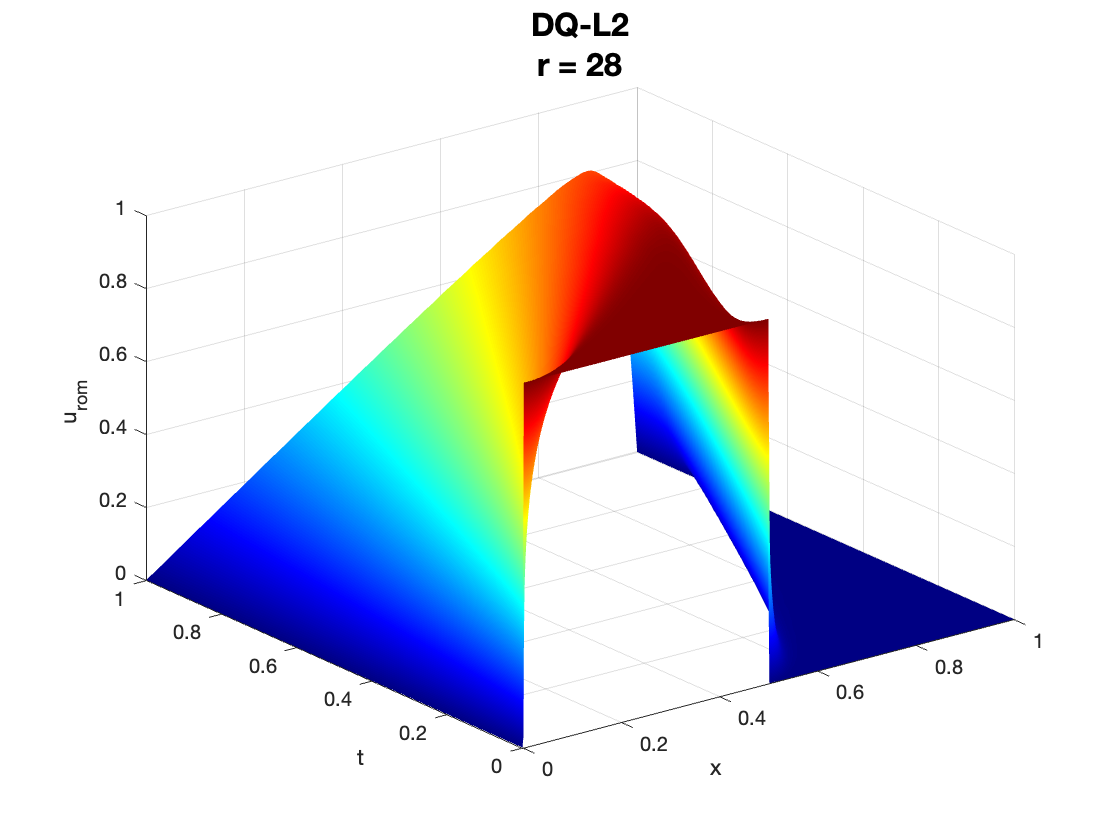}
\includegraphics[width=0.45\textwidth,height=0.3\textwidth]{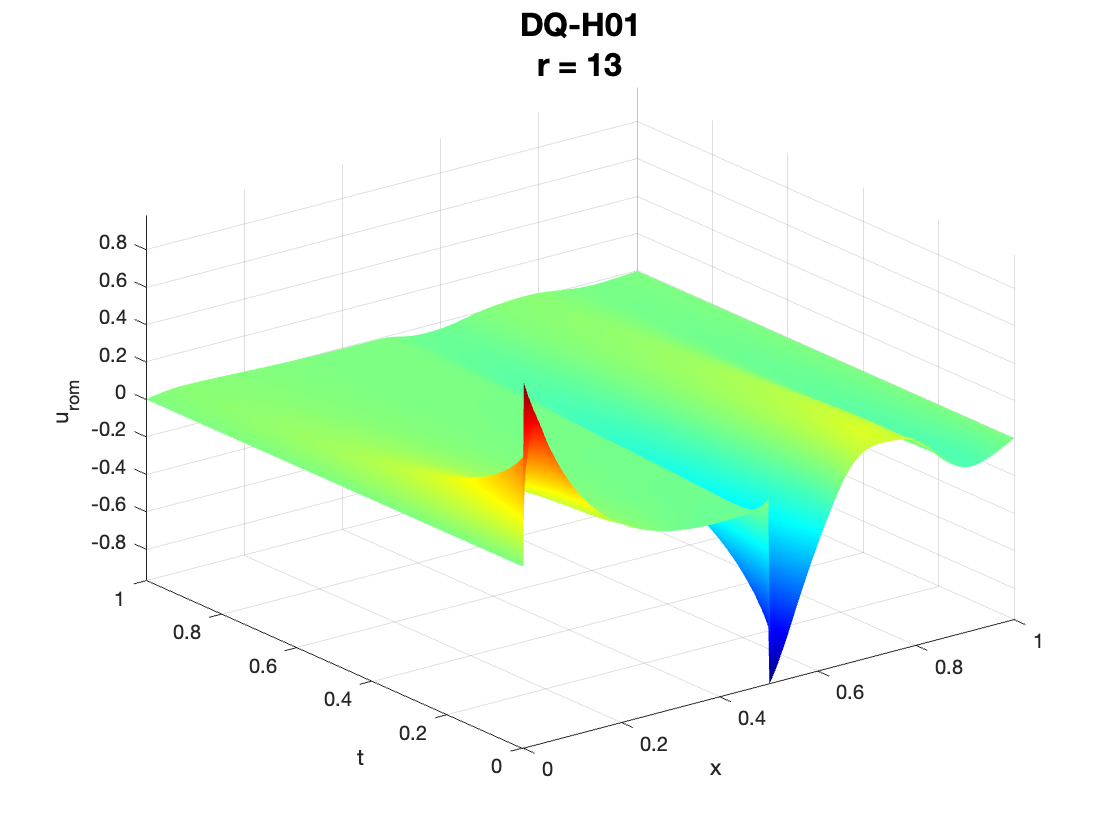}
\includegraphics[width=0.45\textwidth,height=0.3\textwidth]{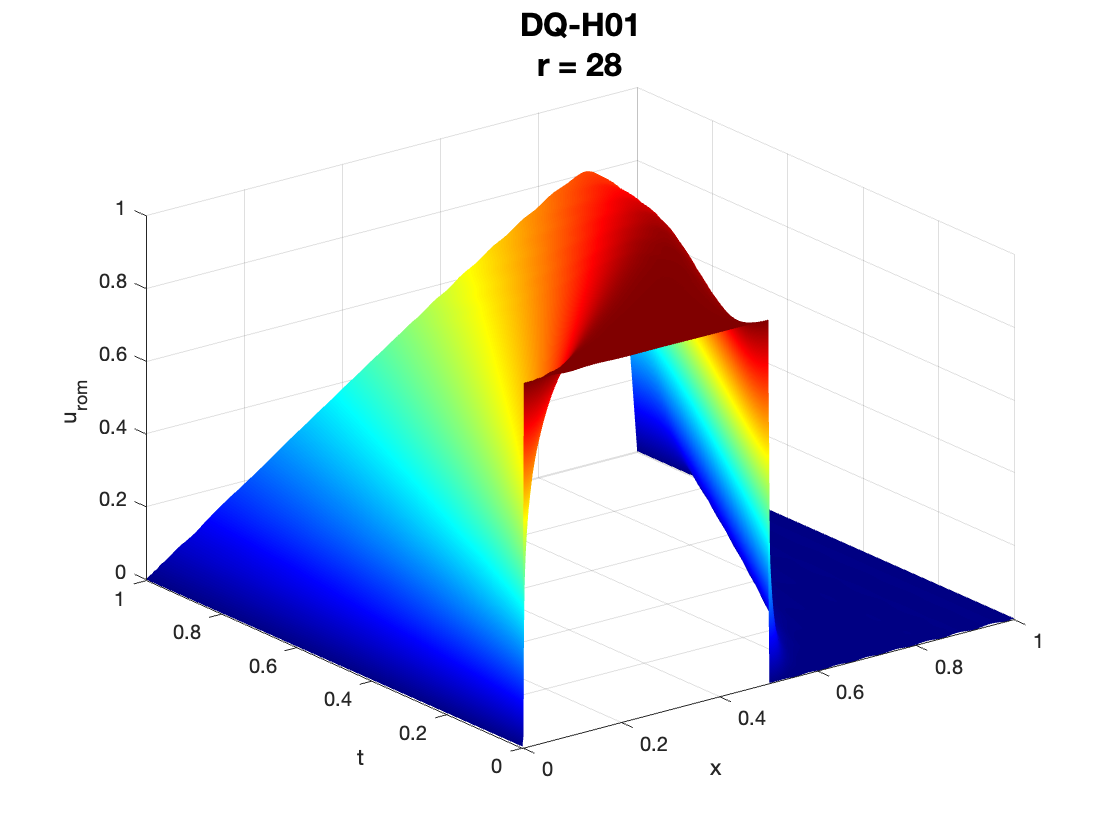}
\caption{Comparison of the DQ-L2 and DQ-H01 solution plots with two different $r$ values. 
} \label{fig:dq_rom_soln}
\end{center} 
\end{figure}

The behavior of the error bounds, which are numerically derived in Sections \ref{sec:nodq_rom_numerics} and \ref{sec:dq_rom_numerics}, are summarized in Table~\ref{table:numerical_convergence_rate}. Considering $L^2$ and $H_0^1$ basis, error norms, and noDQ/DQ frameworks, we observe that the noDQ errors bounds with all cases and the natural norm DQ-L2 and DQ-H01 are optimal; whereas the $l^{\infty}(L^2)$ DQ-L2 and DQ-H01 error bounds are superoptimal.

\subsection{noDQ and DQ ROM Comparison}
\label{sec:nodq_dq_cn_pod_rom_comparison_numerics}
In this section, we numerically compare how all noDQ and DQ ROM errors change based on (i) the ROM dimension, i.e., $r$, (ii) the residual eigenvalues, i.e., $r_{\lambda}$, and (iii) the ratio of residual eigenvalues, i.e., $R_{\lambda}$, which are defined as
\begin{subequations}
\begin{align} 
r_{\lambda}(i):= & \sum_{i=r+1}^{d} \lambda_i, \label{eqn:residual_eig} \\
R_{\lambda}(i):= & \Big(\displaystyle \sum_{i=r+1}^{d} \lambda_i \Big) / \Big(\displaystyle \sum_{i=1}^{d} \lambda_i\Big), \label{eqn:ratio_residual_eig}
\end{align}
\end{subequations}
where $\lambda$ represents the noDQ/DQ eigenvalues, which solve \eqref{eqn:eig_problem}, and $d$ represents the number of positive eigenvalues.

In Figure~\ref{fig:nodq_dq_comparison_case3}, we plot all $l^{\infty}(L^2)$ ROM errors in the left plot and all natural-norm errors in the right plot and compare the behavior of these errors with respect to the ratio of energy kept by the ROM $R_{\lambda}$ in \eqref{eqn:ratio_residual_eig}. 
Both plots in Figure~\ref{fig:nodq_dq_comparison_case3} show that the DQ-L2 recovers a much smaller error than the noDQ-L2 for the same ratio of energy considered. Furthermore, the DQ-H01 has a better slope than the noDQ-H01, even if the DQ-H01 decays later, after stagnation. The H01-norm stagnation is possibly due to the hard test problem we have taken with very sharp gradients.

\begin{figure}[h!]  
\begin{center}
\includegraphics[width=0.45\textwidth,height=0.35\textwidth]{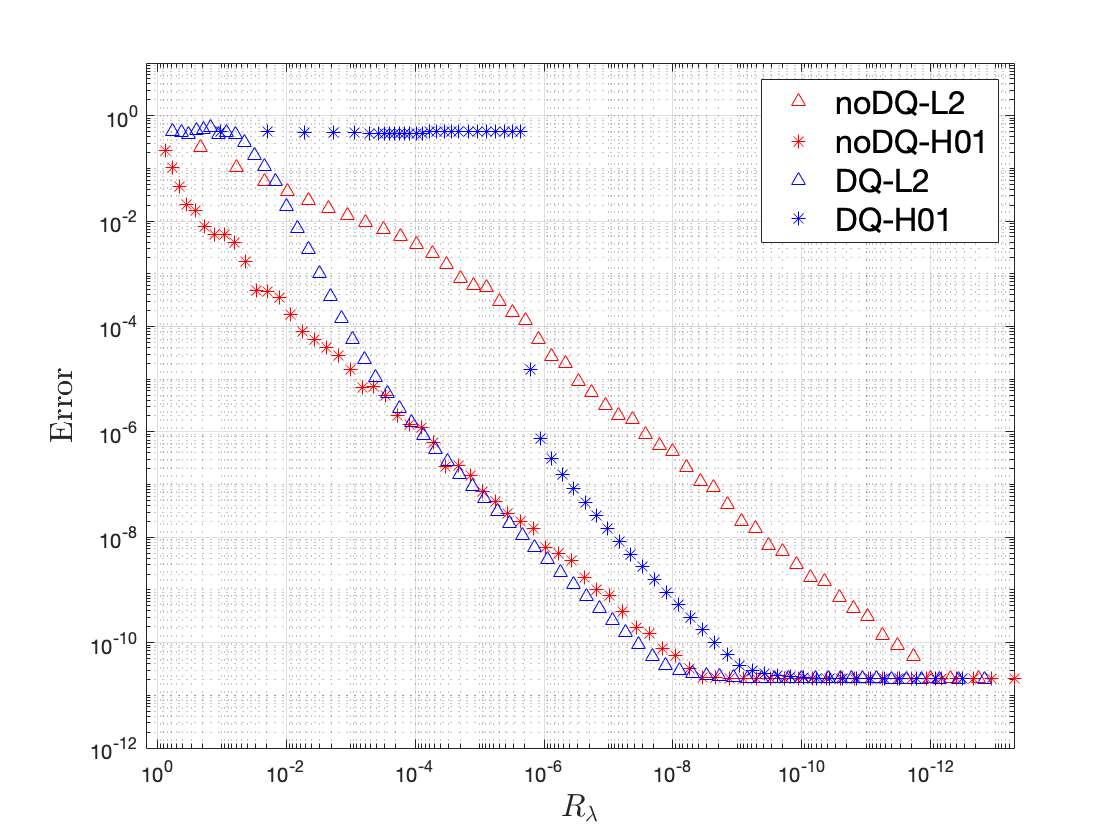}
\includegraphics[width=0.45\textwidth,height=0.35\textwidth]{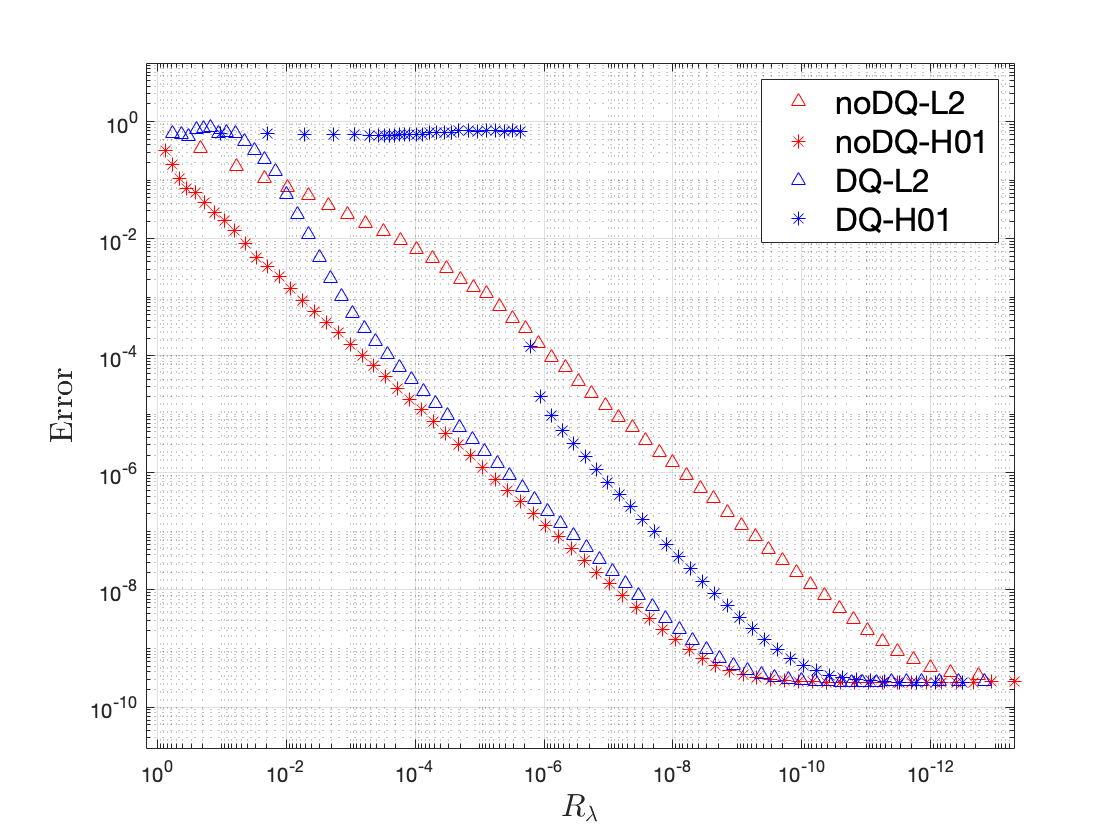}
\caption{$l^{\infty}(L^2)$ (left) and natural-norm (right) ROM errors decay with respect to the ratio of energy kept by the ROM $R_{\lambda}$.
} \label{fig:nodq_dq_comparison_case3}
\end{center} 
\end{figure}

In Figure~\ref{fig:nodq_dq_comparison_case2}, we plot all $l^{\infty}(L^2)$ ROM errors in the left plot and all natural-norm errors in the right plot and compare the behavior of these errors with respect to the energy kept by the ROM $r_{\lambda}$ in \eqref{eqn:residual_eig}. For both plots in Figure~\ref{fig:nodq_dq_comparison_case2}, the DQ ROM errors are over three orders of magnitude smaller than noDQ ROM errors over the same $r_{\lambda}$ interval. Based on these results, we may interpret that for a given energy, the DQ ROMs hold much more information than the noDQ ROMs.

\begin{figure}[h!]  
\begin{center}
\includegraphics[width=0.45\textwidth,height=0.35\textwidth]{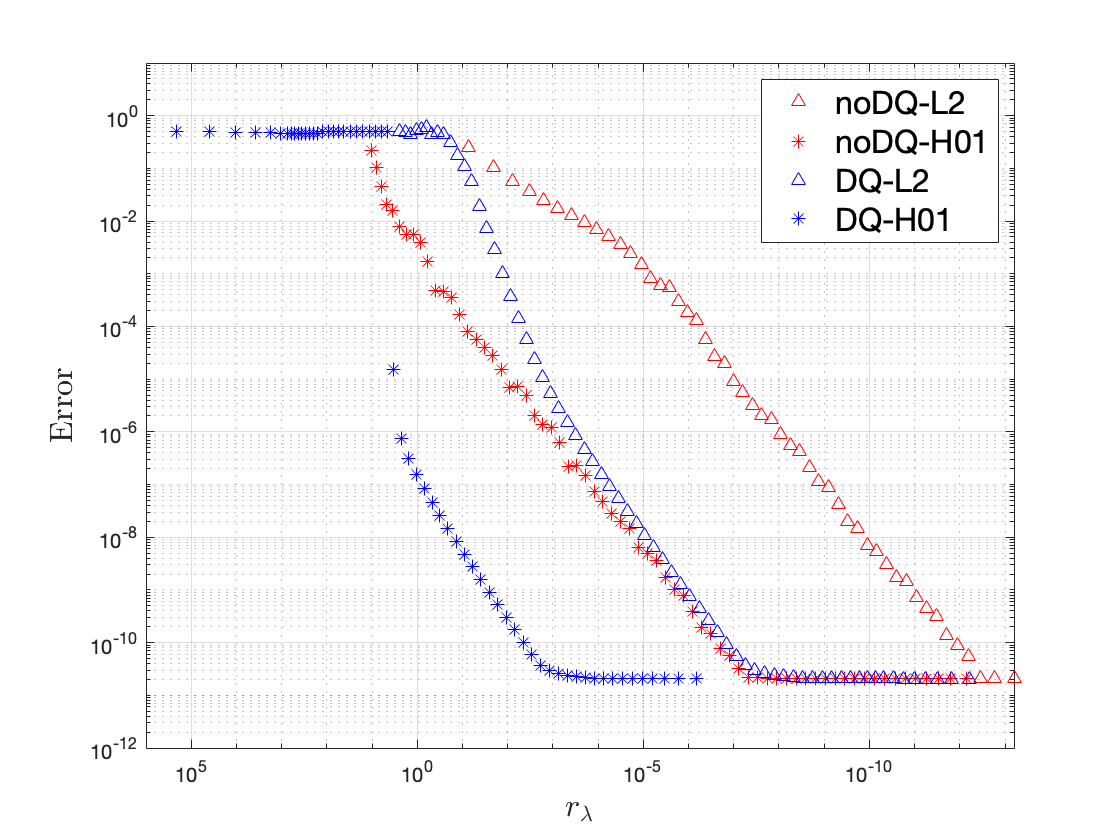}
\includegraphics[width=0.45\textwidth,height=0.35\textwidth]{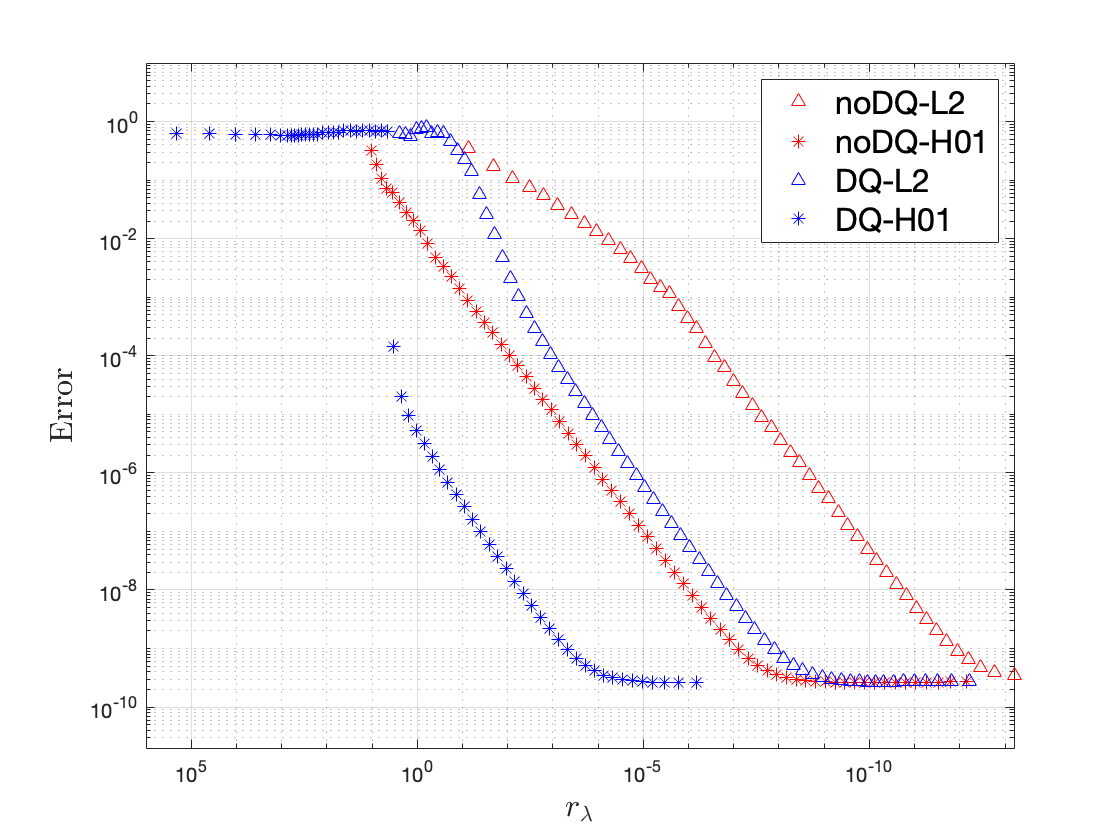}
\caption{$l^{\infty}(L^2)$ (left) and natural-norm (right) ROM errors decay with respect to the energy kept by the ROM $r_{\lambda}$.
} \label{fig:nodq_dq_comparison_case2}
\end{center} 
\end{figure}

In Figure~\ref{fig:nodq_dq_comparison_case1}, we plot all $l^{\infty}(L^2)$ ROM errors in the left plot and all natural-norm errors in the right plot and compare the behavior of these errors with respect to the ROM dimension $r$. 
For almost all $r$ values, in both plots in Figure~\ref{fig:nodq_dq_comparison_case1}, noDQ errors are almost lower than the DQ errors.
The natural-norm noDQ and DQ ROM error behaviors (right plot of Figure~\ref{fig:nodq_dq_comparison_case1}) are more explicit and clearer than the $l^{\infty}(L^2)$ noDQ and DQ ROM error behaviors (left plot of Figure~\ref{fig:nodq_dq_comparison_case1}). 
Furthermore, in both plots in Figure~\ref{fig:nodq_dq_comparison_case1}, we observe that the noDQ-L2, noDQ-H01, and DQ-L2 errors decrease progressively until $r=47$ and $r=55$, in the left and right plot, respectively, then they stagnate. However, the DQ-H01 error in both plots has a constant error behavior until $r=27$, and when it guarantees enough POD modes, it shows a drastic decrease from $r=27$ to $r=28$ and after $r=47$ in the left plot and $r=55$ in the right plot, it stagnates again.

\begin{figure}[h!]  
\begin{center}
\includegraphics[width=0.45\textwidth,height=0.35\textwidth]{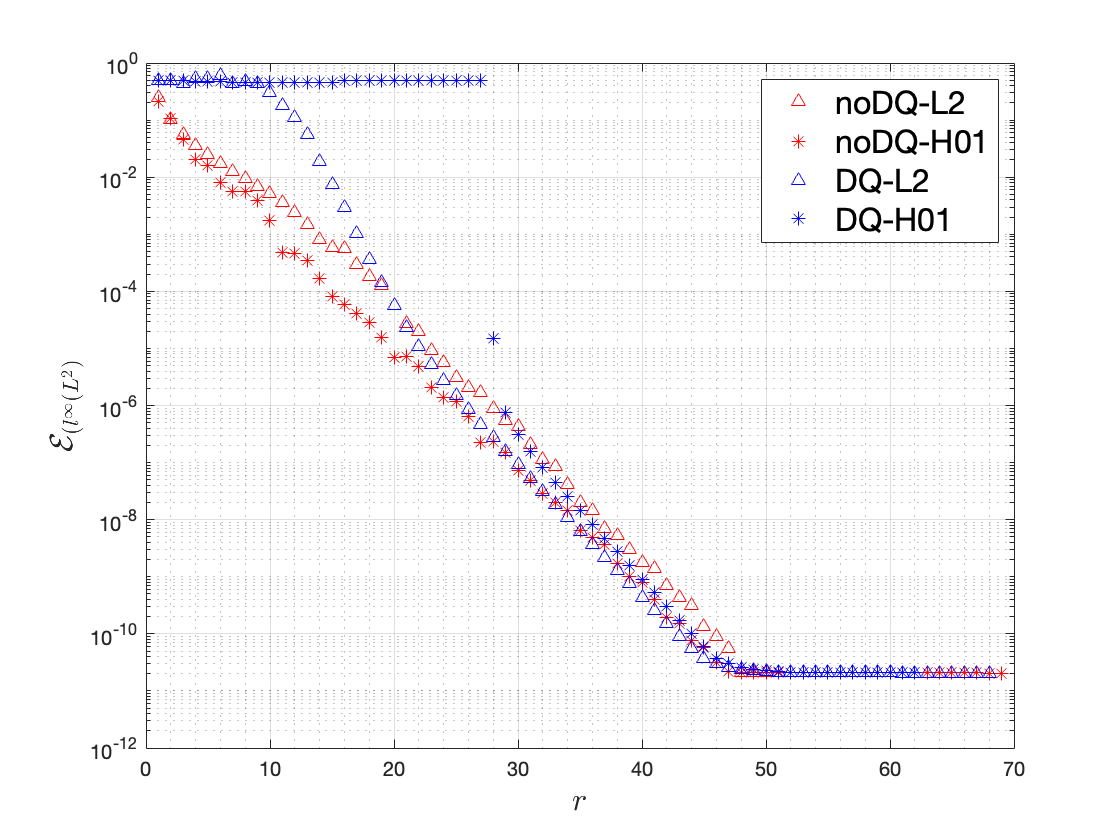}
\includegraphics[width=0.45\textwidth,height=0.35\textwidth]{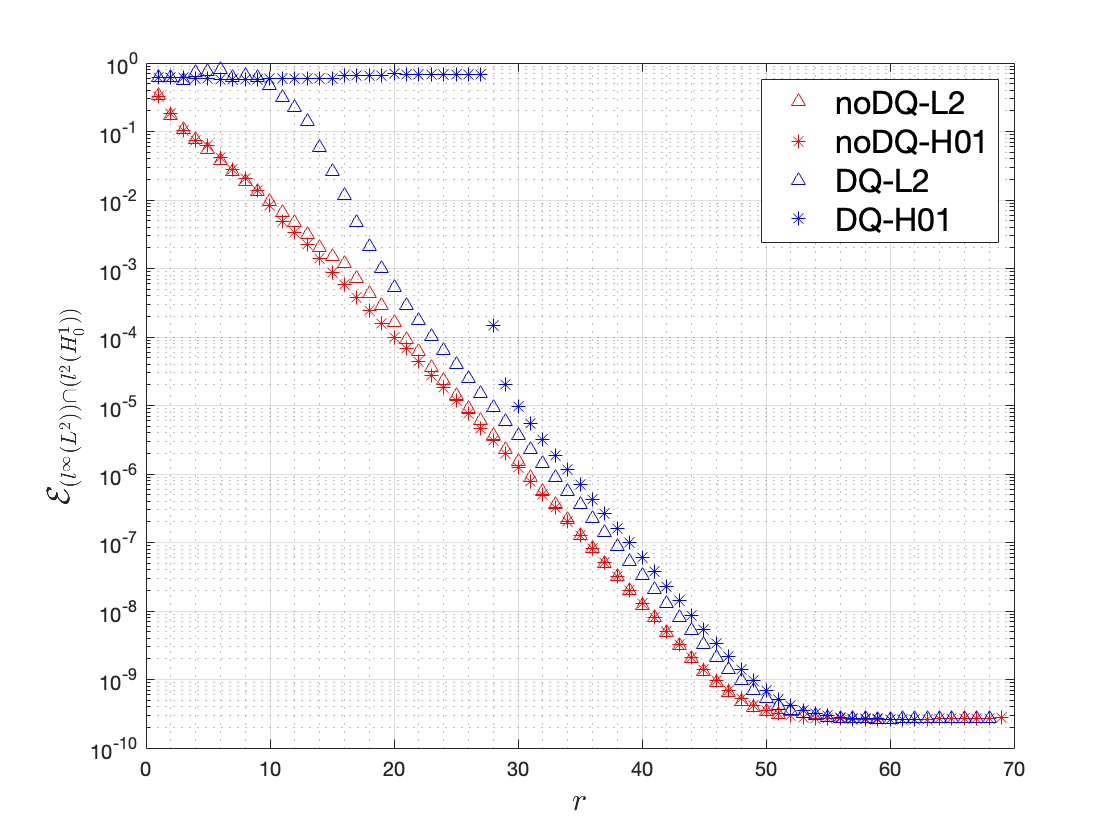}
\caption{$l^{\infty}(L^2)$ (left) and natural-norm (right) ROM errors decay with respect to the ROM dimension $r$.
} \label{fig:nodq_dq_comparison_case1}
\end{center} 
\end{figure}

From Sections~\ref{sec:nodq_rom_numerics} and \ref{sec:dq_rom_numerics}, we numerically observe that $l^{\infty}(L^2)$ noDQ-L2 and noDQ-H01 error bounds are optimal and DQ-L2 and DQ-H01 are superoptimal. If one shows that the slopes of the DQ-L2 and DQ-H01 errors are much sharper than the noDQ-L2 and noDQ-H01 errors, then superoptimal behavior makes sense.
The slope comparison of the errors in Figures~\ref{fig:nodq_dq_comparison_case3} and \ref{fig:nodq_dq_comparison_case2} is clearer than Figure~\ref{fig:nodq_dq_comparison_case1}.

In both Figure~\ref{fig:nodq_dq_comparison_case3} and \ref{fig:nodq_dq_comparison_case2}, 
we observe that the decreasing rates of the DQ-L2 errors are better than the noDQ-L2 ones, being over three orders of magnitude smaller in a wide range of $R_{\lambda}$ and $r_{\lambda}$, respectively, until stagnation occurs, likely due to round-off errors. For the H01-norm comparison, the DQ-H01 has a better slope than the noDQ-H01, even if the DQ-H01 decays later, after stagnation. Therefore, the DQ keeps better-selected information for the same amount of energy due to the difference quotients. This supports the interest in using the DQs.

\section{Conclusions} \label{sec:conclusions}
In this paper, we provided uniform ROM error bounds of nonlinear PDEs, considering the Burgers equation as the first preliminary step considering the DQs. Overall, we theoretically proved and numerically investigated the behavior of the DQ ROM error bounds by considering $L^2(\Omega)$ and $H_0^1(\Omega)$ POD spaces and $l^{\infty}(L^2)$ and natural-norm errors. Furthermore, we provided the noDQ ROM errors without complete theoretical support to make a clear and complete conclusion.

The main results of this paper can be summarized as follows: (i) At the theoretical level, we derived four different DQ ROM error bounds by considering two error norms, i.e., $l^\infty(L^2)$ and natural-norm, and two different POD space frameworks, i.e., $L^2(\Omega)$ and $H_0^1(\Omega)$ POD spaces. 
(ii) In Section~\ref{sec:uniform_max_l2_error}, we theoretically proved that both $l^\infty(L^2)$ DQ-L2 and DQ-H01 errors are optimal with respect to the ROM discretization.
(iii) In Section~\ref{sec:uniform_natural_norm_error}, we obtained the same theoretical results obtained in Section~\ref{sec:uniform_max_l2_error} for the natural-norm DQ-L2 and DQ-H01 errors. 
(iv) In Section~\ref{sec:numerical_results}, we observed that the DQ ROM errors are several orders of magnitude lower than the noDQ errors in terms of the rate of energy kept by the ROM basis.

In Section~\ref{sec:nodq_rom_numerics}, we numerically showed that all noDQ errors have optimal behaviors. The numerical convergence rates of the $l^{\infty}(L^2)$ noDQ-L2 and noDQ-H01 errors are better than the theoretical ones since the $l^{\infty}(L^2)$ noDQ-L2 \eqref{eqn:nodq_l2_error_bound} and noDQ-H01 error \eqref{eqn:nodq_h01_error_bound} are suboptimal theoretical supports.

In Section~\ref{sec:dq_rom_numerics}, we numerically showed that the $l^{\infty}(L^2)$ DQ-L2 and DQ-H01 error bounds are superoptimal; whereas, the natural-norm DQ-L2 and DQ-H01 error bounds are optimal. In Sections~\ref{sec:uniform_max_l2_error} and \ref{sec:uniform_natural_norm_error}, we theoretically proved that all DQ ROM error bounds are optimal. Based on these results, we conclude that the numerical convergence rates for the $l^{\infty}(L^2)$ DQ-L2 and DQ-H01 are better than the theoretical ones.

Furthermore, in Section~\ref{sec:nodq_dq_cn_pod_rom_comparison_numerics}, we provided Figures~\ref{fig:nodq_dq_comparison_case3}-\ref{fig:nodq_dq_comparison_case1} to discuss the superoptimality behavior of the $l^{\infty}(L^2)$ DQ-L2 and DQ-H01 errors and the optimality behavior of the $l^{\infty}(L^2)$ noDQ-L2 and noDQ-H01 errors by comparing their slopes. We believe that considering the time dependency by the DQ inner products increases the linear regression orders.

Finally, in Figures~\ref{fig:nodq_rom_soln} and \ref{fig:dq_rom_soln}, we compared the noDQ-L2 with noDQ-H01 and DQ-L2 with DQ-H01 solution plots to understand how the POD space framework affects the accuracy of the solution for both noDQ and DQ cases. 
In Figure~\ref{fig:nodq_rom_soln}, we concluded that noDQ-L2 and noDQ-H01 yield similar results; however, for a low $r$ value, the noDQ-H01 solution is slightly more accurate than the noDQ-L2. However, the comparison between the DQ-L2 and DQ-H01 is quite different than the noDQ one. 
In Figure~\ref{fig:dq_rom_soln}, we observed that the DQ-H01 yields an accurate solution after enough POD modes are guaranteed but reaching the accurate results is quicker than the DQ-L2.

Extending and improving the effectiveness of the DQs on the Navier-Stokes equations will be our future research direction.

\section*{Acknowledgments}
This work has been supported by Programma Operativo FEDER Andalucia 2014-2020 grant US-1254587 and European Union's Horizon 2020 research and innovation program under the Marie Sklodowska-Curie Actions, grant agreement 872442 (ARIA).

\clearpage
\bibliographystyle{abbrv}
\bibliography{reference}

\end{document}